\documentclass[oneside,12pt]{article}

\usepackage{amssymb, amsmath}
\usepackage{mathtools}
\usepackage[thmmarks,amsmath]{ntheorem}

\usepackage[OT1]{fontenc}

\usepackage[utf8]{inputenc}
\usepackage[T1]{fontenc}
\usepackage{dsfont}
\usepackage{latexsym} 

\usepackage{enumerate}	
\usepackage{footnpag} 
\usepackage{makeidx}  

\usepackage{geometry}	
\usepackage{parcolumns}

\usepackage{graphicx}
\usepackage{color} 

\usepackage{enumitem}
\setlist[enumerate,1]{label=(\roman*)} 

\usepackage[colorlinks = true,
            linkcolor = red,
            urlcolor  = blue,
            citecolor = blue,
            anchorcolor = blue]{hyperref}
\usepackage{cite}

\usepackage[misc]{ifsym} 

\usepackage[OT2,OT1]{fontenc}
\newcommand\cyr
{
\renewcommand\rmdefault{wncyr}
\renewcommand\sfdefault{wncyss}
\renewcommand\encodingdefault{OT2}
\normalfont
\selectfont
}
\DeclareTextFontCommand{\textcyr}{\cyr}

\newcommand\dd{^\textnormal{dd}}
\newcommand\hs{\hspace*{0.1cm}}
\def\mid|{\hs\hs\middle|\hs\hs}

\newcommand{\Sup}{\vee}		
\newcommand{\Inf}{\wedge}	%
\newcommand{\B}{\mathcal{B}}

\newcommand{\I}{\mathcal{I}}

\newcommand\lin{\textnormal{lin}\hspace*{-0.1cm}}

\newcommand{\NN}{\mathbb{N}}	
\newcommand{\RR}{\mathbb{R}} 
\newcommand{\1}{\mathds{1}}					
\newcommand{\sub}{\subseteq}								
\renewcommand\t{\textnormal}

\newcommand{\ohne}{\backslash}
\newcommand{\up}{\uparrow}

\renewcommand{\geq}{\geqslant}	
\renewcommand{\leq}{\leqslant}

\let\int\relax 
\DeclareMathOperator{\int}{int}

\newcommand\eps{\varepsilon}
\renewcommand\phi{\varphi}
\renewcommand\rho{\varrho}

\def\chi{\mathbb{1}}

\renewcommand\c{\cite }
\newcommand{\disp}{\displaystyle}		

\newcounter{Zaehler}
\theoremstyle{plain}
\theoremheaderfont{\normalfont\bfseries}
\theorembodyfont{\itshape}
\theoremseparator{.} 
\theorempreskip{\topsep}
\theorempostskip{\topsep} 
\theoremnumbering{arabic}
\theoremsymbol{}

\newtheorem{theorem}{Theorem}
\newtheorem{lemma}[theorem]{Lemma}
\newtheorem{corollary}[theorem]{Corollary}
\newtheorem{proposition}[theorem]{Proposition}
\newtheorem{definition}[theorem]{Definition}			
\theoremstyle{plain}
\theoremheaderfont{\normalfont\bfseries}
\theorembodyfont{\upshape}
\theoremseparator{.} 
\theorempreskip{\topsep}
\theorempostskip{\topsep} 
\theoremnumbering{arabic}
\theoremsymbol{}

\newtheorem{example}[theorem]{Example}
\newtheorem{numremark}[theorem]{Remark}

\theoremstyle{nonumberplain}
\theoremheaderfont{\itshape}
\theorembodyfont{\normalfont}
\theoremseparator{.} 
\theorempreskip{\topsep}
\theorempostskip{\topsep} 
\theoremsymbol{\ensuremath\square} 
\qedsymbol{$\Box$}

\newtheorem{proof}{Proof}

\begin{document}
\title{Order closed ideals in pre-Riesz spaces and their relationship to bands}
\author{Helena Malinowski\footnote{H. Malinowski, Technische Universit\"at Dresden, Germany \hspace*{2mm}\Letter\hspace*{2mm}lenamalinowski@gmx.de}}

\date{Accepted: January 2019}

\maketitle
\begin{abstract}
In Archimedean vector lattices bands can be introduced via three different coinciding notions. First, they are order closed ideals. Second, they are precisely those ideals which equal their double disjoint complements. The third concept is that of an ideal which contains the supremum of any of its bounded subsets, provided the supremum exists in the vector lattice. We investigate these three notions and their relationships in the more general setting of Archimedean pre-Riesz spaces.
We introduce the notion of a supremum closed ideal, which is related to the third aforementioned notion in vector lattices.
We show that for a directed ideal $I$ in a pervasive pre-Riesz space with the Riesz decomposition property these three concepts coincide, provided the double disjoint complement of $I$ is directed. In pervasive pre-Riesz spaces every directed band is supremum closed and every supremum closed directed ideal $I$ equals its double disjoint complement, provided the double disjoint complement of $I$ is directed.
In general, in Archimedean pre-Riesz spaces the three notions differ. For this we provide appropriate counterexamples.
\end{abstract}


\section{Introduction}
In Archimedean vector lattices bands can be introduced in three different ways. Classically, they are defined as order closed ideals. Moreover, bands are precisely those linear subspaces, which equal their double disjoint complement. A third notion of a band can be found in \c[Definition~17.1(iv)]{Zaa1}. There, a band is defined as an ideal containing the supremum of any of its bounded subsets, whenever the supremum exists in the vector lattice. In Archimedean vector lattices these three notions are equivalent. We investigate the relationships of these notions in Archimedean pre-Riesz spaces.

Pre-Riesz spaces are precisely those (partially) ordered vector spaces that can be order densely embedded into vector lattices. They were introduced 1993 in \c{vanHaa} by van Haandel. Later pre-Riesz spaces and structures therein were thoroughly investigated by Kalauch, Lemmens and van Gaans in \c{Kalauch,5,6,1,3,2,4,7} and \c{9}.
The definition of disjointness in pre-Riesz spaces was first given in \c{1}. Based on this notion, bands were introduced as sets which equal their double disjoint complement. In \c{2} the authors investigate ideals and bands and establish the following result.
\setcounter{theorem}{-1}
\begin{theorem}{\textnormal{\cite[Theorem~5.14]{2}}}\label{KavanGaa}
Let $X$ be a pre-Riesz space. Then every band in $X$ is an order closed ideal.
\end{theorem}
This result raises the question, whether and under which conditions the converse is true, i.e.\ under which conditions every order closed (o-closed) ideal is a band. 
Motivated by this question, we introduce in Section 3 supremum closed (s-closed) ideals and characterize s-closed directed ideals in pervasive pre-Riesz spaces. We show that in a pervasive pre-Riesz space every directed band is an s-closed ideal and every s-closed directed ideal $I$ is a band, provided the double disjoint complement $I\dd$ is directed. In an example we see that the condition of $I\dd$ being directed can not be omitted.
In Section 4 we give conditions under which the converse of Theorem~\ref{KavanGaa} is true: We show that in a pervasive pre-Riesz space with RDP every o-closed directed ideal $I$ is a band, provided $I\dd$ is directed. In a pre-Riesz space with RDP every directed band and every o-closed directed ideal is s-closed. In a pervasive pre-Riesz space every s-closed directed ideal is o-closed, provided $I\dd$ is directed. In an example we demonstrate that the condition of $I\dd$ being directed can not be omitted. Several examples show that in pre-Riesz spaces the notions of a band, an o-closed ideal and an s-closed ideal do not coincide, in general.
The following schemata summarize the results. Grey arrows represent the already known statements. Blue arrows indicate implications which are true under the additional condition that the double disjoint complement is directed. Example numbers attached to blue arrows refer to the given counterexamples in case this additional condition is not satisfied. In the last diagram a green arrow means that the implication is true even under the condition of RDP (i.e.\ the pre-Riesz space need not be pervasive).
%
%
\begin{center}

\begin{picture}(0,0)%
\includegraphics{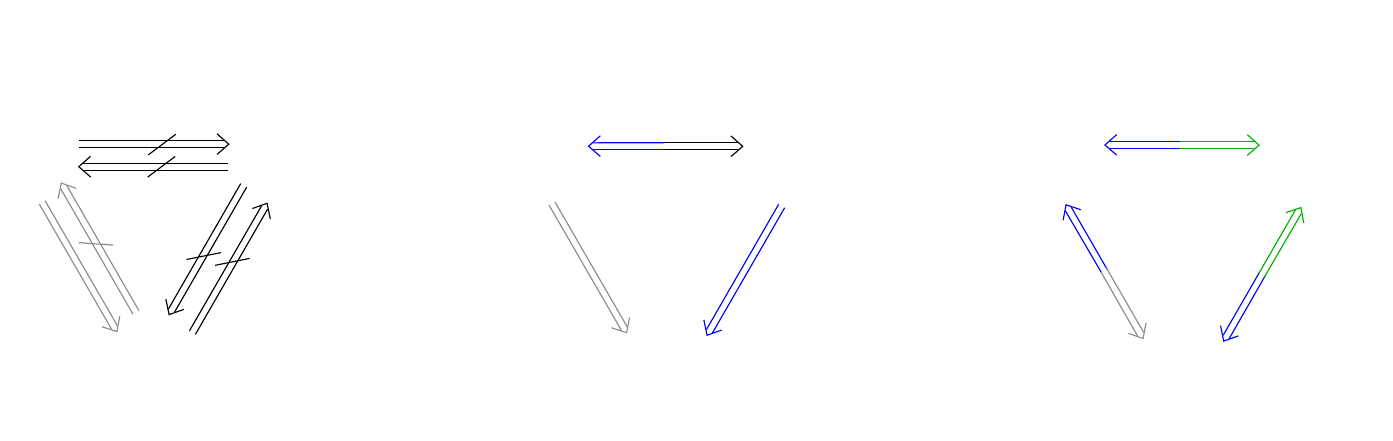}
\end{picture}%
\setlength{\unitlength}{2901sp}%
\begingroup\makeatletter\ifx\SetFigFont\undefined%
\gdef\SetFigFont#1#2#3#4#5{%
  \reset@font\fontsize{#1}{#2pt}%
  \fontfamily{#3}\fontseries{#4}\fontshape{#5}%
  \selectfont}%
\fi\endgroup%
\begin{picture}(9089,2859)(656,-2007)
\put(1708,-410){\makebox(0,0)[b]{\smash{{\SetFigFont{8}{9.6}{\familydefault}{\mddefault}{\updefault}{\color[rgb]{0,0,0}\tiny{Ex.\hspace{.02cm}\ref{ex}}}%
}}}}
\put(1667, 47){\makebox(0,0)[b]{\smash{{\SetFigFont{8}{9.6}{\familydefault}{\mddefault}{\updefault}\tiny{Ex.\hspace{0.02cm}\ref{bands_not_s-closed}}}}}}
\put(5041,570){\makebox(0,0)[b]{\smash{{\SetFigFont{8}{9.6}{\familydefault}{\mddefault}{\updefault}{\color[rgb]{0,0,0}pervasive pre-Riesz}%
}}}}
\put(4186,-195){\makebox(0,0)[b]{\smash{{\SetFigFont{8}{9.6}{\familydefault}{\mddefault}{\updefault}{\color[rgb]{0,0,0}band}%
}}}}
\put(8371,570){\makebox(0,0)[b]{\smash{{\SetFigFont{8}{9.6}{\familydefault}{\mddefault}{\updefault}{\color[rgb]{0,0,0}pervasive pre-Riesz with RDP}%
}}}}
\put(7455,-671){\rotatebox{300.0}{\makebox(0,0)[b]{\smash{{\SetFigFont{8}{9.6}{\familydefault}{\mddefault}{\updefault}{\color[rgb]{0,0,0}\tiny{Thm.\hspace{.02cm}\ref{2.1}}}%
}}}}}
\put(4280,-932){\rotatebox{300.0}{\makebox(0,0)[b]{\smash{{\SetFigFont{8}{9.6}{\familydefault}{\mddefault}{\updefault}{\color[rgb]{0,0,0}\tiny{\c[5.14]{2}}}%
}}}}}
\put(5356, 30){\makebox(0,0)[b]{\smash{{\SetFigFont{8}{9.6}{\familydefault}{\mddefault}{\updefault}{\color[rgb]{0,0,0}\tiny{Cor.\hspace*{0.02cm}\ref{1.4}}}%
}}}}
\put(4636, 30){\makebox(0,0)[b]{\smash{{\SetFigFont{8}{9.6}{\familydefault}{\mddefault}{\updefault}{\color[rgb]{0,0,0}\tiny{Thm.\hspace*{0.02cm}\ref{1.3}}}%
}}}}
\put(8011, 37){\makebox(0,0)[b]{\smash{{\SetFigFont{8}{9.6}{\familydefault}{\mddefault}{\updefault}{\color[rgb]{0,0,0}\tiny{Thm.\hspace{.02cm}\ref{1.3}}}%
}}}}
\put(4703,-328){\makebox(0,0)[b]{\smash{{\SetFigFont{8}{9.6}{\familydefault}{\mddefault}{\updefault}{\color[rgb]{0,0,0}\tiny{Ex.\hspace*{0.02cm}\ref{Namioka1}}}%
}}}}
\put(8675,-1046){\rotatebox{60.0}{\makebox(0,0)[b]{\smash{{\SetFigFont{8}{9.6}{\familydefault}{\mddefault}{\updefault}{\color[rgb]{0,0,0}\tiny{Ex.\hspace*{.02cm}\ref{Namioka2}}}%
}}}}}
\put(5394,-891){\rotatebox{60.0}{\makebox(0,0)[b]{\smash{{\SetFigFont{8}{9.6}{\familydefault}{\mddefault}{\updefault}{\color[rgb]{0,0,0}\tiny{Ex.\hspace*{.02cm}\ref{Namioka2}}}%
}}}}}
\put(856,-188){\makebox(0,0)[b]{\smash{{\SetFigFont{8}{9.6}{\familydefault}{\mddefault}{\updefault}{\color[rgb]{0,0,0}band}%
}}}}
\put(1711,570){\makebox(0,0)[b]{\smash{{\SetFigFont{8}{9.6}{\familydefault}{\mddefault}{\updefault}{\color[rgb]{0,0,0}pre-Riesz}%
}}}}
\put(1036,-953){\rotatebox{300.0}{\makebox(0,0)[b]{\smash{{\SetFigFont{8}{9.6}{\familydefault}{\mddefault}{\updefault}{\color[rgb]{0,0,0}\tiny{\c[5.14]{2}}}%
}}}}}
\put(2339,-989){\rotatebox{60.0}{\makebox(0,0)[b]{\smash{{\SetFigFont{8}{9.6}{\familydefault}{\mddefault}{\updefault}{\color[rgb]{0,0,0}\tiny{Ex.\hspace*{.02cm}\ref{ex1}}}%
}}}}}
\put(2656,-188){\makebox(0,0)[b]{\smash{{\SetFigFont{8}{9.6}{\familydefault}{\mddefault}{\updefault}{\color[rgb]{0,0,0}s-closed}%
}}}}
\put(1666,-1673){\makebox(0,0)[b]{\smash{{\SetFigFont{8}{9.6}{\familydefault}{\mddefault}{\updefault}{\color[rgb]{0,0,0}o-closed}%
}}}}
\put(5986,-195){\makebox(0,0)[b]{\smash{{\SetFigFont{8}{9.6}{\familydefault}{\mddefault}{\updefault}{\color[rgb]{0,0,0}s-closed}%
}}}}
\put(4996,-1680){\makebox(0,0)[b]{\smash{{\SetFigFont{8}{9.6}{\familydefault}{\mddefault}{\updefault}{\color[rgb]{0,0,0}o-closed}%
}}}}
\put(8371,-1718){\makebox(0,0)[b]{\smash{{\SetFigFont{8}{9.6}{\familydefault}{\mddefault}{\updefault}{\color[rgb]{0,0,0}o-closed}%
}}}}
\put(9361,-233){\makebox(0,0)[b]{\smash{{\SetFigFont{8}{9.6}{\familydefault}{\mddefault}{\updefault}{\color[rgb]{0,0,0}s-closed}%
}}}}
\put(7561,-233){\makebox(0,0)[b]{\smash{{\SetFigFont{8}{9.6}{\familydefault}{\mddefault}{\updefault}{\color[rgb]{0,0,0}band}%
}}}}
\put(7817,-1288){\rotatebox{300.0}{\makebox(0,0)[b]{\smash{{\SetFigFont{8}{9.6}{\familydefault}{\mddefault}{\updefault}{\color[rgb]{0,0,0}\tiny{\c[5.14]{2}}}%
}}}}}
\put(8144,-304){\makebox(0,0)[b]{\smash{{\SetFigFont{8}{9.6}{\familydefault}{\mddefault}{\updefault}{\color[rgb]{0,0,0}\tiny{Ex.\hspace*{0.02cm}\ref{Namioka1}}}%
}}}}
\put(9133,-990){\rotatebox{60.0}{\makebox(0,0)[b]{\smash{{\SetFigFont{8}{9.6}{\familydefault}{\mddefault}{\updefault}{\color[rgb]{0,0,0}\tiny{Cor.\hspace{.02cm}\ref{2.2} \ref{2.2.it2}, \ref{2.2.it1}}}%
}}}}}
\put(5724,-1023){\rotatebox{60.0}{\makebox(0,0)[b]{\smash{{\SetFigFont{8}{9.6}{\familydefault}{\mddefault}{\updefault}{\color[rgb]{0,0,0}\tiny{Cor.\hspace{.02cm}\ref{2.2} \ref{2.2.it2}}}%
}}}}}
\put(8731, 37){\makebox(0,0)[b]{\smash{{\SetFigFont{8}{9.6}{\familydefault}{\mddefault}{\updefault}{\color[rgb]{0,0,0}\tiny{Cor.\hspace{.02cm}\ref{2.2cor}}}%
}}}}
\put(1421,-728){\rotatebox{300.0}{\makebox(0,0)[b]{\smash{{\SetFigFont{8}{9.6}{\familydefault}{\mddefault}{\updefault}\tiny{\c[\hspace*{-0.02cm}Ex.\hspace{0.01cm}5.13]{2}}}}}}}
\put(1891,-712){\rotatebox{60.0}{\makebox(0,0)[b]{\smash{{\SetFigFont{8}{9.6}{\familydefault}{\mddefault}{\updefault}{\color[rgb]{0,0,0}\tiny{Ex.\hspace*{.02cm}\ref{Namioka2}}}%
}}}}}
\end{picture}%

\end{center}
%
%

\section{Preliminaries}
Let $X$ be a real vector space and let $K$ be a \emph{cone} in $X$, that is, $K$ is a wedge ($x,y\in K$ and $\lambda,\mu\geq 0$ imply $\lambda x + \mu y \in K$) and we have $K \cap (-K) =\left\{0\right\}$. In $X$ a partial order is introduced by defining $x\leq y$ if and only if $y-x\in K$. Denote by $X_+$ the set of positive elements in $X$. Then $X_+=K$. The pair $(X,\leq)$ is called a (partially) \emph{ordered vector space}. We write loosely $X$ instead of $(X,\leq)$.

Let $X$ be an ordered vector space. An element $u\in X$ is called an \emph{order unit}, if for every $x\in X$ there is some $\lambda\in\RR_{\geq 0}$ such that $\pm x\leq \lambda u$. The space $X$ is called \emph{Archimedean} if for every $x,y\in X$ with $nx\leq y$ for all $n\in\NN$ one has $x\leq 0$. A subset $A$ of $X$ is \emph{directed} if for every  $x,y\in A$ there are $z_1,z_2\in A$ such that $z_1\leq x,y\leq z_2$. A linear subspace $A$ of $X$ is directed if and only if for every $x,y\in A$ there is some $z\in A$ such that $x,y\leq z$. The space $X$ is directed if and only if the cone $X_+$ is generating in $X$, i.e.\ $X = X_+ - X_+$. The space $X$ has the \emph{Riesz decomposition property} (\emph{RDP}) if for every $x_1,x_2,z \in X_+$ with $z \leq x_1+x_2$ there exist $z_1,z_2 \in X_+$ such that $z = z_1+z_2$ with $z_1 \leq x_1$ and $z_2\leq x_2$. Equivalently, $X$ has the RDP if and only if for every $x_1,x_2,x_3,x_4\in X$ with $x_1,x_2 \leq x_3,x_4$ there exists some $z\in X$ such that $x_1, x_2 \leq z \leq x_3, x_4$. 
A set $M\sub X$ is called \emph{bounded above}, if there exists some $z\in X$ such that for every $x\in M$ one has $x\leq z$. Analogously we define \emph{bounded below}. A set is called \emph{order bounded}, if it is bounded above and below.
For $a,b\in X$ with $a\leq b$ we define \emph{order intervals} by $[a,b]:=\left\{x\in X\mid|a\leq x\leq b\right\}$, $\left]a,b\right]:=\left\{x\in X\mid| a< x\leq b\right\}$ and similarly $\left[a,b\right[$ and $\left]a,b\right[$.
A net $(x_\alpha)_\alpha$ in $X$ is said to be \emph{decreasing} (in symbols $x_\alpha\downarrow$), whenever $\alpha\leq\beta$ implies $x_\alpha \geq x_\beta$. For $x\in X$ the notation $x_\alpha \downarrow x$ means that $x_\alpha \downarrow$ and $\inf_\alpha x_\alpha = x$. The symbols $x_\alpha \uparrow$ and $x_\alpha \uparrow x$ are defined analogously. We say that a net \emph{order converges} (short \emph{o-converges}) to $x\in X$ (in symbols  $x_\alpha \stackrel{o}{\rightarrow} x$), if there is a net $(z_\alpha)_\alpha$ in $X$ such that $z_\alpha\downarrow 0$ and for every $\alpha$ one has $\pm (x-x_\alpha) \leq z_\alpha$. The equivalence of $x_\alpha \stackrel{o}{\rightarrow} x$ and $x-x_\alpha \stackrel{o}{\rightarrow} 0$ is obvious. If a net o-converges, then its limit is unique. A set $M\sub X$ is called \emph{order closed} (short \emph{o-closed}) if for every net $(x_\alpha)_\alpha$ in $M$ which o-converges to $x\in X$ one has $x\in M$.
For standard notations in case that $X$ is a vector lattice see \c{PosOp}. Recall that a vector lattice is \emph{Dedekind complete} whenever every non-empty subset which is bounded above has a supremum.

By a \emph{subspace} of an ordered vector space or a vector lattice we mean an arbitrary linear subspace  with the inherited order. We do not require it to be a lattice or a sublattice. 
We call a linear subspace $X_0$ of an ordered vector space $X$ \emph{order dense} in $X$ if for every $x\in X$ we have $x = \inf\left\{z\in X_0 \mid| x\leq z\right\}$, that is, each $x$ is the greatest lower bound of the set $\left\{z\in X_0 \mid| x\leq z\right\}$ in $X$, see \c[p.~360]{113}.
Recall that a linear map $i:X\rightarrow Y$ between two ordered vector spaces $X$ and $Y$ is called \emph{bipositive} if for every $x\in X$ one has $x \geq 0$ if and only if $i(x) \geq 0$. 
An embedding map is linear and bipositive, which implies injectivity. 
If there exists a vector lattice $Y$ and a bipositive linear map $i: X\rightarrow Y$ such that $i(X)$ is order dense in $Y$, then we call $X$ a \emph{pre-Riesz space} and $(Y,i)$ a \emph{vector lattice cover} of $X$.
Vector lattice covers are not unique, in general. For an intrinsic definition of a pre-Riesz space, see \c[Definition~1.1]{vanHaa}.
By \c[Theorem~17.1]{vanHaa} every Archimedean directed ordered vector space is pre-Riesz and every pre-Riesz space is directed. By \c[Proposition~1.4.7]{Kalauch} every vector lattice cover of an Archimedean pre-Riesz space is Archimedean. By \c[Theorem IV.11.1]{Vulikh_en} every Archimedean ordered vector space $X$ has a (unique up to isomorphism) \emph{Dedekind completion} $X^\delta$, i.e.\ a Dedekind complete vector lattice cover of $X$.
Let $Y$ be an Archimedean directed ordered vector space and $X$ an order dense subspace of $Y$. By \c[Theorem~4.14]{vanHaa} the Dedekind completion of $X$ and the Dedekind completion of $Y$ are order isomorphic vector lattices, i.e.\ we can identify $X^\delta = Y^\delta$. In particular, for an Archimedean pre-Riesz space $X$ with a vector lattice cover $Y$ we have $X^\delta = Y^\delta$.

Let $X$ be a pre-Riesz space and $(Y,i)$ a vector lattice cover of $X$. For $M\sub X$ let $M^u$ be the set of all upper bounds of $M$, i.e.\ $M^u=\left\{x\in X\mid| \forall y\in M\colon y \leq x\right\}$. If $s\in X$ is an upper bound of $M$, then we loosely write $M\leq s$. The elements $x,y\in X$ are called \emph{disjoint} (in symbols $x\perp y$) if $\left\{x+y,-x-y\right\}^u = \left\{x-y,-x+y\right\}^u$, for motivation and details see \c{1}. If $X$ is a vector lattice, then this notion of disjointness coincides with the usual one, see \c[Theorem~1.4(4)]{PosOp}. 
By \c[Proposition~2.1(ii)]{1} we have $x\perp y$ if and only if $i(x)\perp i(y)$. The \emph{disjoint complement} of a subset $A\sub X$ is $A^{\t{d}} := \left\{x\in X \mid| \forall a\in A\colon x\perp a\right\}$. A linear subspace $B$ of $X$ is called a \emph{band} if $(B^{\t{d}})^{\t{d}} = B$. The disjoint complement $A^{\t{d}}$ is a band, see \c[Proposition~5.5]{1}. By $\B_A$ we denote the \emph{band generated by} $A$, i.e.\ $\B_A:=A\dd$. If $X$ is an Archimedean vector lattice, then this notion of a band coincides with the classical notion of a band (i.e.\ an o-closed ideal).
The following notion of an ideal was introduced in \c[Definition~3.1]{vanGaa}. A subset $M$ of 
$X$ is called \emph{solid} if for every $x\in X$ and $y\in M$ the relation
$\left\{x,-x\right\}^u \supseteq \left\{y,-y\right\}^u$ implies $x\in M$.
A solid subspace of $X$ is called an \emph{ideal}. If $X$ is a vector lattice, this notion of an ideal coincides with the classical definition. A set $M$ is called \emph{solvex} if for every $x\in X$, $x_1,\ldots,x_n\in M$ the inclusion
\[\left\{x,-x\right\}^u \supseteq \left\{\sum_{i=1}^n \eps_i \lambda_i x_i \mid| \lambda_1,\ldots,\lambda_n\in]0,1],\hs \sum_{i=1}^n\lambda_i =1,\hs \eps_1,\ldots,\eps_n\in\left\{1,-1\right\}\right\}^u\]
implies $x\in M$. Every solvex set is solid and convex, see \c[Lemma~2.3]{9}. By \c[Theorem 4.11]{7} every directed ideal is solvex. By \c[Theorem~5.14]{2} every band in $X$ is an o-closed solvex ideal.
Let $Y$ be a vector lattice. For a set $A\sub Y$ we denote by $\I_A$ the ideal generated by $A$, i.e.\ the smallest ideal in $Y$ containing $A$.

Let $X$ be a pre-Riesz space and $(Y,i)$ a vector lattice cover of $X$. For a set $S\sub Y$ let the preimage of $S$ be denoted by $[S]i:=\left\{x\in X\mid| i(x)\in S\right\}$.
In \c{2} the following \emph{restriction property (R)} and \emph{extension property (E)} for a property $P$ are considered:
\begin{tabbing}
\quad\=(R) \= If $J\sub Y$ has property $P$ in $Y$, then $[J\cap i(X)]i$ has property $P$ in $X$.\\
\quad\=(E) \= If $I\sub X$ has property $P$ in $X$, then there exists a subset $J\sub Y$ with\\ 
      \>\> property $P$ such that $i(I) = J\cap i(X)$ in $Y$.
\end{tabbing}
Property $P$ might be the property of being an ideal or a band. By \c[Propositions~5.12, 5.3 and 5.1(iii)]{1} every band has (E) and every ideal and o-closed ideal has (R). By \c[Proposition 17 (a)]{6} for a band $B$ an extension band is given by $\B_{i(B)}$. By \c[Propositions~5.5(i) and 5.6]{1} every solvex ideal has both (E) and (R). For a solvex ideal $I$ an extension ideal is given by $\I_{i(I)}$. In general, bands do not have (R) and ideals and o-closed ideals do not have (E). For more details, see the overview table in \c[p.~603]{1}.
The pre-Riesz space $X$ is called \emph{pervasive in} $Y$, if for every $y\in Y_+$, $y\neq 0$, there exists some $x\in X$ such that $0<i(x) \leq y$. By \c[Proposition~3.3.20]{Kalauch} the space $X$ is pervasive in $Y$ if and only if $X$ is pervasive in any vector lattice cover. Then $X$ is simply called \emph{pervasive}.
The space $X$ is called \emph{fordable in} $Y$, if for every $y\in Y$ there exists a set $S\sub X$ such that $\left\{y\right\}^{\t{d}} = i(S)^{\t{d}}$ in $Y$. By \c[Proposition~3.3.20(ii)]{Kalauch} the space $X$ is fordable in $Y$ if and only if $X$ is fordable in any vector lattice cover. Then $X$ is simply called \emph{fordable}. By \c[Lemma~2.4]{3} every pervasive pre-Riesz space is fordable. In a fordable pre-Riesz space every band has (R), see \c[Proposition~2.5 and Theorem~2.6]{3}.

All ideals and bands considered here are assumed to be directed, if not stated otherwise. All ordered vector spaces and vector lattices are assumed to be Archimedean.

\smallskip
The following two lemmata were shown in the master's thesis~\c{Waaij}. For the sake of completeness we give here short proofs\footnote{Lemma~\ref{0.0} was originally formulated for the more special case of integrally closed pre-Riesz spaces. For the definition and details see \c{Waaij}.}.
\begin{lemma}{\textnormal{\cite[Theorem~4.15, Corollary~4.16]{Waaij}}}\label{0.0}
Let $X$ be an Archimedean pre-Riesz space and $(Y,i)$ a vector lattice cover of $X$. Then the following statements are equivalent.
\begin{enumerate}
\item\label{properties.1.it1} $X$ is pervasive.
\item\label{properties.1.it2} $\forall a\in X\hs\forall y\in Y\hs \big(i(a)<y \Rightarrow \exists x\in X\colon i(a)<i(x)\leq y\big)$.
\item\label{properties.1.it3} For every $y\in Y_+$ with $y\neq 0$ we have $y = \sup\left(i(X)\cap\left]0, y\right]\right)$.
\item\label{properties.1.it4} For every $y\in Y$ and $z\in X$ with $i(z)< y$ we have
$y = \sup\left(i(X)\cap\left]i(z),y\right]\right)$.
\end{enumerate}
\end{lemma}
\begin{proof}
\ref{properties.1.it1} $\Rightarrow$ \ref{properties.1.it2}: Let $a\in X$ and $y\in Y$ such that $i(a)<y$. Then $y-i(a)\in Y_+$, $y-i(a)\neq 0$. Since $X$ is pervasive, there exists some $z\in X$ with $0 < i(z) \leq y-i(a)$. This yields $i(a)< i(a)+i(z) \leq y$. Thus for $x:=a+z\in X$ we have $i(a)<i(x)\leq y$.

\ref{properties.1.it2} $\Rightarrow$ \ref{properties.1.it3}: Let $y\in X_+$, $y\neq 0$. Then
$M:=i(X)\cap\left]0,y\right]\neq\varnothing$.
We show that $y=\sup M$ by contradiction. Suppose there is an $s'\in Y$ with $M\leq s'$ and $y\not\leq s'$. Then for $s:=s'\Inf y\in Y$ we have $M\leq s$. In particular, $s\in Y_+$ and $s < y$. Then $0<y-s$ and \ref{properties.1.it2} yield that there exists some $u\in X$ with $0< i(u) \leq y-s$. It follows $0<i(u)\leq y$ and thus $i(u)\in M$. As $M\leq s$, we have $0< i(u) \leq s$. We obtain
$0<2i(u)\leq s + (y-s) = y$.
Notice that $2i(u)\in M$ and therefore $0<2i(u)\leq s$. Again, we can combine this with $0< i(u) \leq y-s$ and obtain by induction for every $n\in\NN$ that $n i(u)\leq y$ implies $(n+1)i(u)\leq y$. Altogether, for every $n\in\NN$ we have $0 < n i(u) \leq y$. This is a contradiction to $Y$ being Archimedean. We conclude that $y = \sup\left(i(X)\cap\left]0,y\right]\right)$.

\ref{properties.1.it3} $\Rightarrow$ \ref{properties.1.it4}: Let $y\in Y$ and $z\in X$ such that $i(z)<y$.
Then $y-i(z)>0$ and \ref{properties.1.it3} yield
$y-i(z) = \sup\left(i(X)\cap\left]0,y-i(z)\right]\right)
= \sup \left\{x\in i(X) \mid| i(z)< x+i(z) \leq y\right\}$. By \c[Theorem~13.1]{Zaa1} in an ordered vector space we can interchange addition and the supremum and obtain $y = \sup  \left\{x+i(z)\in i(X) \mid| i(z)< x+i(z) \leq y\right\} = \sup\left(i(X)\cap\left]i(z),y\right]\right)$.

\ref{properties.1.it4} $\Rightarrow$ \ref{properties.1.it1}: This follows immediately, as $i(X)\cap\left]0,y\right]\neq\varnothing$.
\end{proof}

\begin{lemma}{\textnormal{\cite[Theorem~1.44]{Waaij}}}\label{0.1}
Let $X$, $Y$ and $Z$ be ordered vector spaces. Let $X$ be an order dense subspace of $Y$ and $Y$ an order dense subspace of $Z$. Then $X$ is order dense in $Z$.
\end{lemma}
\begin{proof}
Let $z\in Z$. Since $Y$ is order dense in $Z$, we have $z=\inf \left\{y\in Y\mid| z \leq y\right\}$. Due to $X$ being order dense in $Y$, for every $y\in Y$ we have $y=\inf\left\{x\in X\mid| y \leq x\right\}$. It follows
\begin{align}\label{0.1.eq1}
\begin{split}
z 	&= \inf\left\{y \in Y\mid| z\leq y\right\}
	= \inf\left\{\inf\left\{x\in X\mid| y \leq x\right\} \mid| y\in Y, z\leq y\right\} =\\
	&= \inf\left\{x\in X\mid| \exists y\in Y\colon y\leq x\t{ and } z\leq y\right\}.
\end{split}
\end{align}
Due to $X\sub Y$ for every $x\in X$ with $z\leq x$ there exists some $y\in Y$ with $z\leq y\leq x$ (e.g\ $y:=x\in Y$).
It follows $\left\{x\in X\mid| \exists y\in Y\colon z\leq y \t{ and }y\leq x\right\}=\left\{x\in X\mid| z\leq x\right\}$.
Thus due to \eqref{0.1.eq1} the infimum $\inf\left\{x\in X\mid| z\leq x\right\}$ exists in $Z$ and we have $z = \inf\left\{x\in X\mid| z\leq x\right\}$.
That is, $X$ is order dense in $Z$.
\end{proof}

The following technical results will be used later on.
\begin{proposition}\label{1}
Let $X$ be a fordable Archimedean pre-Riesz space and $(Y,i)$ a vector lattice cover of $X$. Let $A\sub X$. Then $[i(A)\dd]i = A\dd$.
\end{proposition}
\begin{proof}
Let $A\sub X$.
Then the set $A\dd$ is a band in $X$. An extension band of $A\dd$ in $Y$ is of the shape $i(A\dd)\dd$.
If we restrict $i(A\dd)\dd$ to $X$, then we obtain $A\dd$, i.e.\ 
\begin{equation}\label{closedness.1000.eq1}
\left[i(A\dd)\dd\right]i = A\dd.
\end{equation}
From $i(A) \sub i(A\dd)$ it follows $i(A)\dd \sub i(A\dd)\dd$.
Restricting this inclusion to $X$ and using \eqref{closedness.1000.eq1} leads to
\begin{equation}\left[i(A)\dd\right]i \sub \left[i(A\dd)\dd\right]i = A\dd.\label{closedness.1000.eq2}\end{equation}
Since $X$ is fordable, we have (R) for bands. If we restrict the band $i(A)\dd$ in $Y$ to $X$, then $\left[i(A)\dd\right]i$ is a band in $X$. Therefore the inclusion $A\sub\left[i(A)\dd\right]i$ yields $A\dd\sub \left(\left[i(A)\dd\right]i\right)\dd = \left[i(A)\dd\right]i$. Together with \eqref{closedness.1000.eq2} this implies $\left[i(A)\dd\right]i = A\dd$.
\end{proof}

\begin{lemma}\label{gelbes_buch}
Let $Y$ be an ordered vector space, $X\sub Y$ a linear subspace and $S\sub X$ be a non-empty subset.
If $\sup_Y S$ exists in $Y$ such that $\sup_Y S\in X$, then $\sup_X S$ exists in $X$ and $\sup_Y S = \sup_X S$.
\end{lemma}
\begin{proof} 
Let $S\sub X$. Let $z:=\sup_Y S$ exist and belong to $X$.
Then $z$ is an upper bound of $S$ in $X$.
Let $a\in X$ be another upper bound of $S$ in $X$. Since in $Y$ we have $z=\sup_Y S$, it follows $z\leq a$ in $Y$. This implies $z\leq a$ in $X$. Thus $z$ is the least upper bound of $S$ in $X$, i.e.\ $z = \sup_X S$.
\end{proof}
Lemma~\ref{gelbes_buch} can be similarly formulated for infima instead of suprema.
Part \ref{properties.23y.it1} of the following corollary was established in \c[Proposition~5.1~(i)]{2}.
\begin{corollary}\label{properties.23y}
Let $X$ be a pre-Riesz space, $(Y,i)$ a vector lattice cover of $X$ and $S\sub X$ a non-epmpty subset.
\begin{enumerate}
\item\label{properties.23y.it1} If $\sup S$ exists in $X$, then $\sup i(S)$ exists in $Y$ and $\sup i(S) = i(\sup S)$.
\item\label{properties.23y.it2} If $\sup i(S)$ exists in $Y$ and $\sup i(S)\in i(X)$, then $\sup S$ exists in $X$ and $\sup i(S) = i(\sup S)$.
\end{enumerate}
\end{corollary}

\section{Supremum closed ideals and their relationship to bands}
In this section we introduce the concept of an s-closed ideal. In Proposition~\ref{1.0} we show that in pervasive pre-Riesz spaces the notion of an s-closed directed ideal is equivalent to the definition of a band given in \c[Chapter 1, \S 4]{Zaa1} by Luxemburg and Zaanen. 

Let $X$ be a pervasive pre-Riesz space and $I$ a directed ideal in $X$.
Theorem~\ref{1.3.5} characterizes positive elements of $I\dd$.
In Theorem~\ref{1.3} we show that if $I$ is s-closed, then $I$ is a band, provided that $I\dd$ is directed. Conversely, Corollary~\ref{1.4} yields that if $I$ is a band, then $I$ is s-closed. In Example~\ref{Namioka1} we show that in Theorem~\ref{1.3} we can not drop the condition of $I\dd$ being directed.

\begin{definition}
Let $X$ be an ordered vector space. An ideal $I$ in $X$ is \textbf{supremum closed} (short \textbf{s-closed}), if for every $z\in X_+$ the relation $z = \sup\left(I\cap [0,z]\right)$ implies $z\in I$.
\end{definition}
The condition $z = \sup\left(I\cap [0,z]\right)$ means that the supremum exists and equals $z$.

In the two subsequent results we characterize s-closed ideals.
\begin{lemma}\label{1.5}
Let $X$ be an ordered vector space and $I$ and ideal in $X$. Then $I$ is s-closed if and only if
$\forall z\in X \hs\big( \left(\exists a\in I, a\leq z:\hs z=\sup\left(I\cap [a,z]\right)\right) \Rightarrow z\in I\big)$.
\end{lemma}
\begin{proof}
``$\Rightarrow$'':
Let $z\in X$. Assume that there exists some $a\in I$ with $a\leq z$ such that $z=\sup\left(I\cap [a,z]\right)$. By \c[Theorem~13.1]{Zaa1} in an ordered vector space we can interchange addition and the supremum. It follows
\[z-a = \sup\left\{x-a\in I \mid| 0\leq x-a\leq z-a\right\} = \sup\left(I\cap [0,z-a]\right).\]
Since $I$ is s-closed, it follows $z-a\in I$. As $a\in I$, we obtain $z = (z-a) +a \in I$.

``$\Leftarrow$'':
The special case $a:=0$ yields that $I$ is s-closed.
\end{proof}

The following alternative definition of a band was given in \c[Chapter 1, \S 4]{Zaa1} by Luxemburg and Zaanen. 
The ideal $I$ in a vector lattice $Y$ is called a \textit{band} if 
for every subset $A\sub I$ we have
\begin{equation}\label{eq_zclosed}
\sup A \t{ exists in } Y\hs\Rightarrow\hs \sup A \in I.
\end{equation}
In a vector lattice this definition is equivalent to the standard definition of a band, i.e.\ to $I$ being an o-closed ideal.
The following result relates \eqref{eq_zclosed} to the s-closedness of a directed ideal in pre-Riesz spaces.
\begin{proposition}\label{1.0}
Let $X$ be an Archimedean pervasive pre-Riesz space and $I\sub X$ a directed ideal. 
Then the following statements are equivalent:
\begin{enumerate}
\item\label{1.0-item1} $I$ is s-closed
\item\label{1.0-item2} $\forall A\sub I:\hs \sup A \t{ exists in } X \Rightarrow \sup A \in I$.
\end{enumerate}
\end{proposition}
\begin{proof}
\ref{1.0-item1} $\Rightarrow$ \ref{1.0-item2}:
Since $X$ is Archimedean, $X$ has the Dedekind completion $(X^\delta,i)$.
For a non-empty subset $A$ of the s-closed and directed ideal $I$ let $s:=\sup A$ exist in $X$.
We show that there exists some $z\in I$ such that $i(s) = \sup i\left(I\cap\left[z,s\right]\right)$.
Then by Corollary~\ref{properties.23y}~\ref{properties.23y.it2} this leads to
$s =\sup \left(I\cap\left[z,s\right]\right)$.
Since $I$ is s-closed, due to Lemma~\ref{1.5} we then conclude that $s\in I$.

To show that there exists some $z\in I$ such that $i(s) = \sup j\left(I\cap\left[z,s\right]\right)$, we proceed as follows. First we approximate $-i(s)^-$ and $i(s)^+$ from below by elements in $i(I)$. That is, we establish
\begin{align}
\exists z\in I,\hs i(z)\leq -i(s)^-\colon -i(s)^- \hs&=\hs \sup \left(i(I)\cap \left[i(z),-i(s)^-\right]\right),\label{gl3_1}\\ 
i(s)^+ \hs&=\hs \sup \left(i(I)\cap \left[0,i(s)^+\right]\right). \label{gl3_2}
\end{align}
Finally, for a certain set $S_2\sub i(I)$ we then establish $\sup j\left(I\cap\left[z,s\right]\right)=\sup S_2$ and $\sup S_2=i(s)$.

To show \eqref{gl3_1}, let $a\in A$. Since $I$ is directed and $a\in I$, there exists an element $z\in I$ such that $z\leq 0, a$. Due to $a\leq s$ it follows $i(z) \leq 0\Inf i(a) \leq 0\Inf i(s) = -i(s)^-$.
Since $X$ is pervasive, by Lemma~\ref{0.0} we obtain
$-i(s)^-=\sup \left(i(X)\cap \left[i(z),-i(s)^-\right]\right)$.
As $i(z)\in i(I)$, for every $x\in i(X)$ with $i(z)\leq x\leq -i(s)^-\leq 0$ we have $x\in i(I)$. This yields \eqref{gl3_1}.

To show \eqref{gl3_2}, let $A^{(+)} :=\left\{i(a)^+\in X^\delta\mid|a\in A\right\}$. We first show for every $u\in A^{(+)}$ that $u=\sup \left(i(I) \cap\left[0, u\right]\right)$. Since $X$ is pervasive, by Lemma~\ref{0.0} for every $u \in A^{(+)}$ we have $u=\sup \left(i(X) \cap\left[0, u\right]\right)$. The ideal $I$ is directed and therefore solvex. Thus $\I_{i(I)}$ is an extension ideal of $I$, i.e.\ we have $\left[\I_{i(I)}\right]i=I$. Due to $u\in A^{(+)}\sub \I_{i(I)}$ it follows for every $x\in X$ with $i(x)\in[0,u]$ that $x\in I$. This implies $u=\sup \left(i(I) \cap\left[0, u\right]\right)$. We obtain the following chain of inequalities, which establishes \eqref{gl3_2}:
\begin{align*}
\begin{split}
i(s)^+ \hs&=\sup A^{(+)} =
\sup\left\{ \sup\left(i(I)\cap [0,u]\right) \mid| u\in A^{(+)} \right\}\hs\leq\\
&\leq\hs \sup \left\{\sup\left(i(I)\cap [0,i(s)^+]\right) \mid| u\in A^{(+)}\right\} \hs=\hs\\
&=\hs\sup\left(i(I)\cap [0,i(s)^+]\right)\hs\leq\hs i(s)^+.
\end{split}
\end{align*}

Let $S_1:=i\left(I\cap[z,s]\right)$ and $S_2 :=\left\{x_1+x_2 \in i(I)\hspace*{.5mm}\middle|\hspace*{.5mm} x_1\in[0, i(s)^+], x_2\in[i(z),-i(s)^-]\right\}$. In the next step we show
\begin{align}\label{gl3_4}
\sup S_1 =\sup S_2.
\end{align}
Clearly, every upper bound of $S_1$ is an upper bound of $S_2$. It is left to show that every upper bound of $S_2$ is an upper bound of $S_1$. Let $u\in X^\delta$ such that $S_2\leq u$.
Then for every $x_1\in i(I)\cap[0, i(s)^+]$ and every $x_2\in i(I)\cap[i(z),-i(s)^-]$ we have $x_1+x_2\leq u$. In the inequality $x_1 \hs\leq\hs u-x_2$ fix the element $x_2$. Then \eqref{gl3_2} leads to $i(s)^+ \hs\leq\hs u - x_2$. We take the supremum over all $x_2$ in the inequality $x_2 \hs\leq\hs u - i(s)^+$ and obtain from \eqref{gl3_1} that $i(s)^+-i(s)^- \hs\leq\hs u$. Since for every $x\in S_1$ we have $x\leq i(s)^+-i(s)^-=i(s)$, it follows $x\leq u$. Thus $S_1\leq u$ and we obtain \eqref{gl3_4}.

Next, we show
\begin{equation}\label{gl3_5}
\sup S_2 =i(s).
\end{equation}
Clearly, $\sup S_2 \leq i(s)$. To see the converse inequality, let $u\in X^\delta$ be an upper bound of $S_2$. We show $u\geq i(s)$. For every $x_1,x_2\in i(I)$ with $x_1\in[0, i(s)^+]$ and $x_2\in[i(z),-i(s)^-]$ we have $u\geq x_1+x_2$. Then \eqref{gl3_2} yields $u\geq i(s)^+ + x_2$ and \eqref{gl3_1} yields $u\geq i(s)^+ + i(s)^-$. This implies
$u \geq\hs i(s)$ and establishes \eqref{gl3_5}.

From \eqref{gl3_4} and \eqref{gl3_5} together it follows
\[i(s)=\sup S_2 = \sup S_1=\sup \left(i\left(I\cap[z,s]\right)\right).\]
Corollary~\ref{properties.23y}~\ref{properties.23y.it2} yields
$s =\sup \left(I\cap\left[z,s\right]\right)$. Since $I$ is s-closed, we get $\sup A=s\in I$.

\ref{1.0-item2} $\Rightarrow$ \ref{1.0-item1}:
Let $z\in X_+$ be such that $z=\sup\left(I\cap[0,z]\right)$. By our hypothesis it follows for $A:=I\cap[0,z]$ that $\sup A=z\in I$. Thus $I$ is s-closed.
\end{proof}

The following technical result will be used in the proofs of Theorems~\ref{1.3} and \ref{1.3.5}.
\begin{lemma}\label{1.2}
Let $X$ be an Archimedean pervasive pre-Riesz space and $I$ a directed ideal in $X$. Then for every $z\in (I\dd)_+$ we have
$z=\sup\left(I\cap[0,z]\right)$.
\end{lemma}
\begin{proof}
Since $X$ is Archimedean, $X$ has the Dedekind completion $(X^\delta,i)$. Let $I$ be a directed ideal in $X$. Then $I$ is solvex and therefore has the extension property. The ideal $\I_{i(I)}$ in $X^\delta$ as an extension ideal of $I$, i.e.\ we have $\I_{i(I)} \cap i(X) = i(I)$. Let $z\in (I\dd)_+$ be such that \mbox{$z\neq 0$}.
Since every pervasive pre-Riesz space is fordable, by Proposition~\ref{1} we obtain $z\in I\dd = \left[i(I)\dd\right]i$, from which it follows
$i(z) \in i(I)\dd \sub \I_{i(I)}\dd$.
By \c[Theorem 3.4]{PosOp} in a vector lattice $Y$ the band $\B_A$ generated by an ideal $A$ has the form $\left\{x\in Y \mid| \exists (x_\alpha)_\alpha \t{ in } A_+: 0\leq x_\alpha\up |x|\right\}$. Since $i(z)$ is a positive element in the band $\I_{i(I)}\dd$, there exists a net $(y_\alpha)_\alpha$ in $(\I_{i(I)})_+$ such that $0\leq y_\alpha \up i(z)$.

Let $M:=I\cap\left]0,z\right]$. We show that $\sup i(M) = i(z)$. 
Define for every $y_\alpha\neq 0$ the set
$\widehat{M}_\alpha := i(X)\cap\left]0,y_\alpha\right]$.
Since $I$ is directed, for every $\alpha$ we have $y_\alpha\in\I_{i(I)}$. This yields 
$\widehat{M}_\alpha=i(X)\cap\left]0,y_\alpha\right] \sub i(X)\cap\I_{i(I)}=i(I)$.
Thus for every $\alpha$ we have
$\widehat{M}_\alpha =i(I)\cap\left]0,y_\alpha\right]$ and, consequently, $\widehat{M}_\alpha\sub i(M)$.

Due to $X$ being pervasive by Lemma~\ref{0.0} we have $y_\alpha = \sup \widehat{M}_\alpha$. Thus in $X^\delta$ we obtain
\begin{equation*}
y_\alpha = \sup \widehat{M}_\alpha \leq \sup \bigcup_\alpha \widehat{M}_\alpha \leq \sup i(M) \leq i(z).
\end{equation*}
Taking supremum over all $\alpha$ yields
$\sup i(M)=i(z)\in i(X)$. By Corollary~\ref{properties.23y}~\ref{properties.23y.it2} the supremum $\sup M$ exists in $X$ and $\sup M=z$.
\end{proof}

From the vector lattice theory we know that every o-closed ideal is a band and every band is o-closed. Next we show both implications in pervasive pre-Riesz spaces, considering s-closed ideals instead of o-closed ideals. Notice that in both cases we assume that the ideal is directed and in Theorem~\ref{1.3} we have the condition that its double disjoint complement is likewise directed. These conditions are natural, as all bands and ideals in vector lattices are directed. 
\begin{theorem}\label{1.3}
Let $X$ be an Archimedean pervasive pre-Riesz space and $I$ an s-closed directed ideal in $X$.
Then $(I\dd)_+\sub I_+$.

If, in addition, the band $I\dd$ is directed, then $I=I\dd$, i.e.\ $I$ is a band.
\end{theorem}
\begin{proof}
Let $z\in (I\dd)_+$, $z\neq 0$. Since $X$ is pervasive and $I$ is directed, by Lemma~\ref{1.2} we have $z=\sup\left(I\cap[0,z]\right)$. Due to $I$ being s-closed it follows $z\in I_+$. Additionally, if $I\dd$ is directed, then we have $I\dd=(I\dd)_+-(I\dd)_+ \sub I_+-I_+ = I$.
\end{proof}

\begin{theorem}\label{1.3.5}
Let $X$ be an Archimedean pervasive pre-Riesz space and let $I$ be a directed ideal in $X$. 
Then $\Big\{x\in X_+ \hs\Big|\hs x=\sup \left(I\cap[0,x]\right)\Big\} = (I\dd)_+$.
\end{theorem}
\begin{proof}
As $X$ is pervasive and $I$ is a directed ideal, by Lemma~\ref{1.2} for every $x\in (I\dd)_+$ we have $x= \sup\left(I\cap[0,x]\right)$. It follows
$\Big\{x\in X_+ \hs\Big|\hs x=\sup \left(I\cap[0,x]\right)\Big\} \supseteq (I\dd)_+$.
To see the converse inclusion,
let $x\in X_+$, $x\neq 0$, such that $x=\sup\left(I\cap[0,x]\right)$. Let $y\in I^{\t{d}}$, i.e.\ for every $z\in I$ we have $y\perp z$. In particular, 
\begin{equation}\label{gl6-1}
\forall a\in I\cap [0,x]\colon y\perp a.
\end{equation}
We show that $y\perp x$, then it follows $x\in I\dd$. Suppose, $y\not\perp x$. Then in a vector lattice cover $(Y,i)$ of $X$ we have $|i(y)|\Inf i(x)>0$. Since $X$ is pervasive, there exists an element $\tilde{x}\in X$ with $0<i(\tilde{x}) \leq |i(y)|\Inf i(x)$. That is, we have the two inequalities
\begin{equation}\label{gl6-2}
0<\tilde{x}\leq x \quad\t{ and }\quad 0<i(\tilde{x})\leq |i(y)|.
\end{equation}
Fix an $a\in I\cap [0,x]$. Then $0< a \leq x$ and \eqref{gl6-2} yield
\begin{equation}\label{gl6-3}
0\leq 0\Sup i(\tilde{x})< i(a)\Sup i(\tilde{x}) \leq i(x)\Sup i(\tilde{x}) = i(x).
\end{equation}
From \eqref{gl6-1} and \eqref{gl6-2} it follows $a\perp\tilde{x}$. Since $a\perp \tilde{x}$ in $X$ is equivalent to $i(a)\perp i(\tilde{x})$ in $Y$, we obtain $i(a)\Sup i(\tilde{x}) = i(a) + i(\tilde{x})$.
Then due to \eqref{gl6-3} we have $i(a) + i(\tilde{x}) \leq i(x)$, and so, $a \leq x -\tilde{x}$. Since this is true for every $a\in I\cap [0,x]$, it follows $x = \sup\left(I\cap[0,x]\right) \leq x-\tilde{x}$, a contradiction to $\tilde{x}>0$.
Thus our assumption $y\not\perp x$ is false, i.e.\ we obtain for every $y\in I^{\t{d}}$ that $y\perp x$. It follows $x\in (I\dd)_+$.
\end{proof}

\begin{corollary}\label{1.4}
Let $X$ be an Archimedean pervasive pre-Riesz space. Then every directed band in $X$ is s-closed.
\end{corollary}
\begin{proof}
Let $B$ be a directed band in $X$. In particular, $B$ is a directed ideal. 
By Theorem~\ref{1.3.5} it follows $\Big\{x\in X_+ \hs\Big|\hs x=\sup\left(I\cap[0,x]\right)\Big\} = (B\dd)_+ = B_+$.
\end{proof}

The following example shows that in Theorem~\ref{1.3} the condition of $I\dd$ being directed can not be omitted. Note that in Example~\ref{Namioka1} the pre-Riesz space is not only pervasive, but has other strong additional conditions.
\begin{example}\label{Namioka1}
\textit{In an Archimedean pervasive pre-Riesz space with RDP and an order unit a directed s-closed ideal need not be a band.}

Consider the vector space
\[X:=\left\{f\in C([-1,1]) \mid| f(0) = \frac{f(-1)+f(1)}{2}\right\}\]
endowed with pointwise order. The space is Archimedean as a subspace of $C([-1,1])$ and directed, since it contains the order unit $\1_{[-1,1]}$. Thus $X$ is pre-Riesz. It was shown in \c[V.2, Example 9]{VulikhWeber} that $X$ is not a vector lattice, but has the RDP. By \c[Example 18]{6} the space $C([-1,1])$ is a vector lattice cover of $X$. We show that $X$ is pervasive. Let $y\in C([-1,1])$ be such that $y> 0$. Then there exists some $p\in[-1,1]$ and a neighbourhood $U(p)$ of $p$ such that $y(U(p))>0$. Thus we can find an $x\in C([-1,1])$ such that $x(-1)=x(0)=x(1)=0$ and $0<x(t)\leq y(t)$ for every $t$ in a subset of $U(p)\ohne\left\{-1,0,1\right\}$. It follows $0<x\leq y$ and $x\in X$, i.e.\ $X$ is pervasive.

The linear subspace
\[I:=\left\{ f\in X \mid| f([-\textstyle{\frac{1}{2}},0]\cup\left\{-1,1\right\}) = 0 \right\}\]
of $X$ is an ideal. Moreover, it is even a vector lattice with the pointwise order inherited from $X$. However, $I$ is not a band. Let namely
\[g(t):=\begin{cases}
		2t+1 \quad & t\in\left[-1,-\frac{1}{2}\right[,\\
		0 & t\in\left[-\frac{1}{2},0\right],\\
		t & t\in\left]0,1\right].
\end{cases}\]
Then $g \in X$. Since functions in $X$ are continuous, we have
\[I^\t{d}=\left\{f\in X \mid| f\left(\left]-1,-\textstyle{\frac{1}{2}}\right]\cup\left[0,1\right[\right) = 0\right\}=\left\{f\in X \mid| f\left(\left[-1,-\textstyle{\frac{1}{2}}\right[\cup\left]0,1\right]\right) = 0\right\}.\]
It follows $I\dd=\left\{f\in X \mid| f\left(\left[-\textstyle{\frac{1}{2}},0\right]\right) = 0\right\}$.
Clearly, $g\in I\dd$. However, due to $g(-1)=-1$, $g(1)=1$ we have $g\notin I$, i.e.\ the ideal $I$ is not a band.

We show that $I$ is s-closed. Let $s\in X_+$, $s\neq 0$, be such that $s=\sup \left(I\cap[0,s]\right)$. We have to establish that $s\in I$.
We first show $s\left(\left[-\tfrac{1}{2},0\right]\right)=0$.
Suppose there exists some $t\in\hs]-\frac{1}{2},0[$ with $s(t)>0$. Due to $s$ being continuous we can find an open neighbourhood $U(t)\sub\hs]-\frac{1}{2},0[$ of $t$, where $s$ is strictly positive. Thus there exists a continuous function $a\in X_+$, $a\neq 0$, with support in $U(t)$ and $a(t)< s(t)$. The function $\tilde{s}:=s-a>0$ belongs to $X$.
Since on the interval $]-\frac{1}{2},0[$ the elements of $I\cap[0,s]$ vanish, we have for every $x\in I\cap[0,s]$ that $x\leq \tilde{s}$. This is a contradiction to $s$ being the supremum of the set $I\cap[0,s]$. We obtain $s\left(\left]-\frac{1}{2},0\right[\right)=0$. Due to $s\in X$ being continuous, it follows $s(-\frac{1}{2})=s(0)=0$. However, $s\geq 0$ and $0=s(0) =\frac{s(-1)+s(1)}{2}$ imply $s(-1)=s(1)=0$. Altogether, we have $s([-\frac{1}{2},0]\cup\left\{-1,1\right\})=0$, i.e.\ $s\in I$.
We conclude that $I$ is an s-closed ideal.

We show that $I\dd$ is not directed. Indeed, we have $g,-g\in I\dd$. However, if $f\in X$ is such that $g,-g\leq f$, then $1=-g(-1)\leq f(-1)$ and $1=g(1)\leq f(1)$. It follows $f(0) = \frac{f(-1)+f(1)}{2} \geq 1$, i.e.\ $f(0)\neq 0$. We obtain $f\notin I\dd$. Thus $I\dd$ is not directed.
\end{example}

\section{Relationship between order closed and supremum closed ideals}
In this section we focus on o-closed directed ideals and its relationship to s-closed ideals and to bands.
Recall that in an Archimedean vector lattice an ideal $I$ is o-closed if and only if $I$ is a band. In Theorem~\ref{2.1} we generalize this fact for pervasive pre-Riesz spaces with RDP.
In Corollary~\ref{2.2} we show that in pre-Riesz spaces with RDP o-closed directed ideals are s-closed. Moreover, in pervasive pre-Riesz spaces every s-closed directed ideal $I$ is o-closed, provided $I\dd$ is directed. Example~\ref{Namioka2} shows that we can not drop the condition that $I\dd$ is directed.
We consider other examples, which rule out most of the remaining questions about the relationship of bands, o-closed ideals and s-closed ideals without additional assumptions on the pre-Riesz space.

\smallskip
To establish Theorem~\ref{2.1} we first need the following lemma.
\begin{lemma}\label{closedness.13a}
Let $X$ be an Archimedean pre-Riesz space with RDP and $I\sub X$ an ideal. Then for every $x\in X_+$ the set $I\cap [0,x]$ is directed.
\end{lemma}
\begin{proof}
Let $X$ have the RDP and let $I\sub X$ be a directed ideal. Let $x\in X_+$. If $x=0$, then the statement is clear. Let $x>0$. Consider two elements $a_1,a_2\in I\cap [0,x]$. Since $I$ is directed, there exists an element $a_3\in I$ such that $a_1,a_2\leq a_3$. Since $X$ has the RDP, the relationship $a_1,a_2\leq x,a_3$ implies that there exists some $a\in X$ such that $a_1,a_2\leq a\leq x,a_3$. Due to $0\leq a\leq a_3\in I$ we have $a\in I$. Thus we established that for every $a_1,a_2\in I\cap [0,x]$ there exists an element $a\in I\cap [0,x]$ such that $a_1,a_2\leq a$, i.e.\ $I\cap [0,x]$ is upward directed. Due to $0\in I\cap [0,x]$ it is immediate that $I\cap [0,x]$ is downward directed. 
\end{proof}

An example of a pre-Riesz space which satisfies all assumptions of the subsequent Theorem~\ref{2.1}, but is not a vector lattice, can be found in \c[Example~1.58]{AliTou}. There it is shown that the space $C^1[0,1]$ of all differentiable functions on the interval $[0,1]$ has the RDP but fails to be a vector lattice. It is clear that $C^1[0,1]$ is pervasive in its vector lattice cover $C[0,1]$.
Recall that in vector lattices ideals and bands are automatically directed. Furthermore every Archimedean vector lattice is a pervasive pre-Riesz space and has the RDP. 
Therefore the following Theorem~\ref{2.1} is a proper generalization of the fact that in an Archimedean vector lattice an ideal $I$ is o-closed if and only if $I=I\dd$.

\begin{theorem}\label{2.1}
Let $X$ be an Archimedean pervasive pre-Riesz space with RDP and $I$ an o-closed directed ideal in $X$. If $I\dd$ is directed, then $I$ is a band, i.e.\ $I = I\dd$.

In particular, then $I= \left\{x\in X \mid| \t{there exists a net } (x_\alpha)_\alpha \t{ in } I,\hs x_\alpha \xrightarrow{o} x\right\}$.
\end{theorem}
\begin{proof}
We only have to show $I\dd\sub I$. Let $x\in I\dd$. We establish that there exists a net $(x_\alpha)_\alpha$ in $I$ with $x_\alpha \xrightarrow{o}x$. Since $I$ is o-closed, we then conclude that $x\in I$. As $I\dd$ is directed, for every $x\in I\dd$ there are $a, b\in (I\dd)_+$ such that $x=a-b$. Thus it is sufficient to find for every $a\in (I\dd)_+$ a net $(a_\alpha)_\alpha$ in $I$ with $a_\alpha \xrightarrow{o}a$.

Let $a\in(I\dd)_+$ and $M:=I\cap [0,a]$. We use the set $M \sub X$ as the index set for the net $(a_\alpha)_\alpha$ in $I$. 
Since $X$ has the RDP, by Lemma~\ref{closedness.13a} the set $M$ is directed. Since $X$ is pervasive, by Lemma~\ref{1.2} we have $a=\sup M$.
For every $\alpha \in M$ let $a_\alpha:=\alpha$. Then the net $(a_\alpha)_{\alpha\in M}$ is upward directed and due to $\sup M =a$ we have $a_\alpha \uparrow a$. Let $z_\alpha := a - a_\alpha$. Then $0 \leq a - a_\alpha = z_\alpha \downarrow 0$, i.e.\ the net $(a_\alpha)_{\alpha\in M}$ order converges to $a$.
\end{proof}
It is open whether Theorem~\ref{2.1} remains true if we drop the pervasiveness or the RDP or if we do not assume $I\dd$ to be directed.
The following example shows that in Theorem~\ref{2.1} we can not drop simultaneously the conditions that $X$ is pervasive and has the RDP.
\begin{example}\label{ex}
\textit{In an Archimedean pre-Riesz space an o-closed and s-closed directed ideal $I$, for which the double disjoint complement $I\dd$ is directed, need not be a band.}

Let $X:=\left\{p\in C(\RR)\mid|p \t{ polynomial of degree at most } 2\right\}$ with pointwise order. Then $X$ is Archimedean and directed and therefore a pre-Riesz space. The space $X$ is not pervasive, since there does not exist a non-zero positive element in $X$ which is less or equal both, $q(t)=t^2$ and $c(t)=1$, where $t\in\RR$. For the same reason $X$ does not have the RDP.

Consider the subspace $I:=\left\{\lambda q \mid| \lambda\in\RR\right\}$ of $X$. Note that for every $f\in I$ we have $f(0)=0$ and $f'(0)=0$. It follows that $I$ is a directed ideal.
Moreover, $I$ is o-closed. Indeed, a net $(x_\alpha)_\alpha = (\lambda_\alpha q)_\alpha$ of elements in $I$ o-converges to an $x\in X$ if and only if the net $(\lambda_\alpha)_\alpha$ converges. In particular, if $(\lambda_\alpha q)_\alpha$  o-converges, then its limit is an element of the form $\lambda q\in I$. Thus $I$ is o-closed. Furthermore, $I$ is s-closed. Indeed, let $A\sub\RR$. The subset $\left\{\lambda q\mid| \lambda\in A\right\}$ of $I$ is order bounded if and only if $A$ is a bounded subset of $\RR$. Let $A\sub\RR$ be bounded. Then the supremum of the set $\left\{\lambda q\mid| \lambda\in A\right\}$ exists in $X$ and equals the element $(\sup A)\cdot q \in I$. To sum up, $I$ is a directed o-closed and s-closed ideal.
The double disjoint complement $I\dd=X$ of $I$ is directed. Clearly, $I$ is not a band.
\end{example}

The following corollary relates s-closedness and o-closedness of directed ideals.
\begin{corollary}\label{2.2}
Let $X$ be an Archimedean pre-Riesz space and $I\sub X$ a directed ideal.

\begin{enumerate}
\item\label{2.2.it1} Let, in addition, $X$ have the RDP. If $I$ is o-closed, then $I$ is s-closed.
\item\label{2.2.it2} Let, in addition, $X$ be pervasive and the band $I\dd$ be directed. If $I$ is s-closed, then $I$ is o-closed.
\end{enumerate}
\end{corollary}

\begin{proof}
\ref{2.2.it1} 
Let $I\sub X$ be a directed o-closed ideal and let $x\in X_+$ be such that $x=\sup(I\cap[0,x])$. Then by Lemma~\ref{closedness.13a} the set $I\cap[0,x]$ is directed. Set $x_\alpha:=\alpha$ for every $\alpha \in I\cap[0,x]$. Then for the net $(x_\alpha)_\alpha$ in $I$ we have $x_\alpha \uparrow x$. Due to $X$ being o-closed, $x$ belongs to $I$. Thus $I$ is s-closed.

\ref{2.2.it2} Let $I\sub X$ be a directed and s-closed ideal. Since the band $I\dd$ is directed, by Theorem~\ref{1.3} the ideal $I$ coincides with $I\dd$ and is therefore a band.
By Theorem~\ref{KavanGaa} every band is automatically o-closed.
\end{proof}

From Theorem~\ref{KavanGaa} and Corollary~\ref{2.2}~\ref{2.2.it1} we obtain the following. 
\begin{corollary}\label{2.2cor}
Let $X$ be an Archimedean pre-Riesz space with RDP. Then every directed band in $X$ is s-closed.
\end{corollary}

Next example shows that in Corollary~\ref{2.2}~\ref{2.2.it2} we can not omit the additional condition of $I\dd$ being directed.
\begin{example}\label{Namioka2}
\textit{In an Archimedean pervasive pre-Riesz space with RDP and an order unit a directed s-closed ideal need not be o-closed.}

We return to Example~\ref{Namioka1}. There we considered the ordered vector space
$X=\left\{f\in C([-1,1]) \mid| f(0) = \frac{f(-1)+f(1)}{2}\right\}$, which is an Archimedean pervasive pre-Riesz space with RDP and an order unit. Let the ideal $I$ and the function $g$ be as in Example~\ref{Namioka1}.
We already established that $I$ is s-closed and a vector sublattice of $X$. Moreover, $I\dd$ is not directed.
We show that $I$ is not o-closed. For every $n\in\NN_{\geq 2}$ define functions
\[g_n (t):=\begin{cases}
		-tn-n & t\in[-1,-1+\textstyle{\frac{1}{n+2}}], \\
		g(t) & t\in\hs]-1+\textstyle{\frac{1}{n+2}},1-\textstyle{\frac{1}{n+1}}], \hs\hs\qquad \t{ and}\\
		-nt+n & t\in\hs ]1-\textstyle{\frac{1}{n+1}},1] \\
\end{cases}\]
\[u_n(t):=\begin{cases}
g(t)-g_n(t) & t\in[-1,-1+\textstyle{\frac{1}{n+2}}]\cup[1-\textstyle{\frac{1}{n+1}},1],\\
nt+1 & t\in[-\textstyle{\frac{1}{n}},0],\\
-nt+1& t\in[0,\textstyle{\frac{1}{n}}],\\
0 & \t{elsewhere}.
\end{cases}
\]
For every $n\in\NN_{\geq 2}$ we have $g_n \in I$ and $u_n \in X$. A straightforward calculation yields
$\pm(g-g_n)\leq u_n$ and $u_n\downarrow 0$.
Thus the sequence $(g_n)_n$ in $I$ order converges to $g\notin I$, i.e.\ $I$ is not o-closed.
\end{example}

It is open whether in pervasive pre-Riesz spaces o-closed directed ideals are s-closed. The next example shows that in arbitrary pre-Riesz spaces this is not true.
\begin{example}\label{ex1}
\textit{In an Archimedean pre-Riesz space an o-closed directed ideal need not be s-closed.}

For the interval $A:=[0,1]\cup\left\{2\right\}\sub\RR$ consider the vector space $B(A)$ of bounded functions on $A$ and a linear subspace of $B(A)$ given by
\[X=\left\{\lambda \1_A\mid| \lambda\in\RR\right\}\oplus\lin\left\{\1_{\left\{t,2\right\}} \mid| t\in[0,1]\right\}.\]
Let $X$ be ordered pointwise. The characteristic function $\1_{A}$ is an order unit in $X$, hence $X$ is directed. Since $B(A)$ is Archimedean, $X$ is likewise Archimedean. Thus $X$ is a pre-Riesz space.
However, $X$ is not a vector lattice, as for $x_1:=\1_{\left\{0,2\right\}}$ and $x_2:=\1_{\left\{1,2\right\}}$ there does not exist a supremum in $X$. Indeed, $s_1:=x_1+x_2\in X$ and $s_2:=\1_{A}\in X$ are both upper bounds of $x_1$ and $x_2$. Suppose, $u\in X$ is supremum of $x_1$ and $x_2$. Then we have $x_1,x_2\leq u\leq s_1, s_2$. It follows $u(0)=u(1)=u(2)=1$ and $u(t)=0$ for every $t\in A\ohne\left\{0,1,2\right\}$. However, such a function $u$ can not be written as a linear combination of elements in $X$, i.e.\ $u\notin X$, a contradiction. Notice that this also establishes that $X$ does not have the RDP.

Clearly, $B(A)$ is a vector lattice. We show that $X$ is order dense in $B(A)$. To that end, let $f\in B(A)$. We can assume $f\geq 0$ or else translate $f$ by adding the function $n\1_{A}\in X$ for a sufficiently large $n\in\NN_{\geq 0}$. To write $f$ as an infimum of a subset of $X$, we define a family of functions as follows. Let $a, b \in[0,1]$ be such that $a\neq b$. Consider the following three cases.

Case 1: If $f(a)<f(2)$, then set $f_a$ to be the function
\[f_a:=||f||_\infty \1_A \hs +\hs (f(a)-||f||_\infty)\1_{\left\{a,2\right\}} \hs +\hs (f(2)-f(a))\1_{\left\{b,2\right\}}.\]
We have $f_a\in X$, $f_a(a)=f(a)$ and $f_a(2)=f(2)$. Furthermore, $f_a(b)=||f||_\infty+f(2)-f(a)>||f||_\infty$ and $f_a(t)=||f||_\infty\geq f(t)$ for $t\in A\ohne\left\{a,b,2\right\}$, i.e.\ $f_a\geq f$ pointwise.

Case 2: If $f(a)=f(2)$, then set $f_a$ to be the function
\[f_a:=||f||_\infty \1_A + (f(a)-||f||_\infty)\1_{\left\{a,2\right\}}.\]
We have $f_a\in X$, $f_a(a)=f(a)$ and $f_a(2)=f(2)$. Furthermore $f_a(t)=||f||_\infty\geq f(t)$ for $t\in A\ohne\left\{a,2\right\}$, i.e.\ $f_a\geq f$ pointwise.

Case 3: If $f(a)>f(2)$, then set $f_a$ to be the function
\[f_a:=2||f||_\infty \1_A \hs +\hs (f(a)-2||f||_\infty)\1_{\left\{a,2\right\}} \hs +\hs (f(2)-f(a))\1_{\left\{b,2\right\}}.\]
We have $f_a\in X$, $f_a(a)=f(a)$ and $f_a(2)=f(2)$. Furthermore, since $f(a)-f(2)\leq f(a)\leq ||f||_\infty$, we have $f_a(b)=2||f||_\infty+f(2)-f(a)\geq ||f||_\infty$ and $f_a(t)=2||f||_\infty\geq f(t)$ for $t\in A\ohne\left\{a,b,2\right\}$, i.e.\ $f_a\geq f$ pointwise.

To sum up, for every $a\in[0,1]$ we found a function $f_a\in X$ such that $f\leq f_a$ and $f_a(a)=f(a)$. Moreover, $f_a(2)=f(2)$. We obtain $f=\inf\left\{f_a \mid| a\in [0,1]\right\}$ in $B(A)$ and therefore $f=\inf\left\{g\in X \mid| f\leq g\right\}$. This establishes that $B(A)$ is a vector lattice cover of $X$.
Notice that $X$ is not pervasive, since there does not exist an element $x\in X$ with $0< x\leq \1_{\left\{2\right\}} \in B(A)$.

Next, we present a directed o-closed ideal $I$ in $X$, which is not s-closed.
Let $S:=\left\{\1_{\left\{x,2\right\}} \mid| x\in [0,1]\right\}\sub X$ and let $I$ be the linear hull of $S$, i.e.\ $I:=\lin\hs S$. It is immediate that $I$ is a directed ideal. Moreover, $I$ is not s-closed, since $\disp\sup\left\{\1_{\left\{t,2\right\}}\mid| t\in[0,1]\right\}=\1_{A} \notin I$. Note that $I^\t{d}=\left\{0\right\}$ and thus $I$ is not a band.

It is left to show that $I$ is o-closed. If an $x\in X$ does not belong to $I$, then it has the form $x=\mu\1_{A}+z$ with $\mu\neq 0$ and $z\in I$. We consider only the element $x=\1_{A}+0$ and establish that there is no net in $I$ which order converges to $\1_{A}+0$. Other elements $x\in X\ohne I$ can be treated similarly. This then implies that the ideal $I$ is o-closed.

Suppose, on the contrary, that a net $(x_\alpha)_{\alpha\in\mathbb{A}}$ in $I$ o-converges to $\1_{A}$. Then the net $(x_\alpha(2))_{\alpha\in\mathbb{A}}$ of reals converges. 
We show that $(x_\alpha(2))_{\alpha}$ diverges to establish a contradiction.
As $x_\alpha\xrightarrow{o} \1_A$, there exists a net $(y_\alpha)_{\alpha\in\mathbb{A}}$ in $X$ such that $\pm(\1_A-x_\alpha)\leq y_\alpha$ and $y_\alpha\downarrow 0$.
For every $\alpha$ the element $x_\alpha$ can be written as a sum of finitely many elements in $I$. That is, for every $\alpha$ there is a finite subset $N_\alpha$ of $[0,1]$ such that
\begin{equation}\label{ex1.eq1}
x_\alpha = \sum_{r\in N_\alpha}x_\alpha(r)\cdot\1_{\left\{r,2\right\}}.
\end{equation}
Fix an $\alpha$. Since for every $t\in[0,1]\ohne N_\alpha$ we have $x_\alpha(t)= 0$, it follows for such $t$ that $\pm(\1_A(t)-x_\alpha(t))=\pm 1\leq y_\alpha(t)$. Since $N_\alpha$ is finite, we conclude that $y_\alpha$ has a representation $y_\alpha=\mu_\alpha \1_A + z_\alpha$ with $\mu_\alpha\in\RR_{\geq 1}$ and $z_\alpha\in I$.
We show
\begin{equation}\label{ex1.eq2}
\forall \eps>0 \hs\forall r\in[0,1]\hs\exists\alpha\in\mathbb{A}\colon y_\alpha(r)\leq \eps.
\end{equation}
Suppose the contrary, i.e.\ 
\begin{equation}\label{ex1.eq3}
\exists\eps>0 \hs\exists r\in[0,1] \hs\forall\alpha\in\mathbb{A}\colon y_\alpha(r)\geq \eps.
\end{equation}
For an $\tilde{r}\in[0,1]\ohne\left\{r\right\}$ define the element $s=\eps \1_{\left\{r,2\right\}}-\eps\1_{\left\{\tilde{r},2\right\}}\in X$. Then for every $\alpha$ we have $s \leq y_\alpha$. That is, $s$ is a lower bound of the set $\left\{y_\alpha\mid| \alpha\in\mathbb{A}\right\}$. Since $s\not\leq 0$ we obtain a contradiction to $y_\alpha\downarrow 0$. Thus our assumption \eqref{ex1.eq3} is false. From \eqref{ex1.eq2} and  $y_\alpha\downarrow$ it follows
\begin{equation}\label{ex1.eq4}
\forall\eps> 0\hs \forall r\in[0,1]\hs \exists \alpha\in\mathbb{A}\hs \forall \beta\geq \alpha\colon y_\beta(r)\leq\eps.
\end{equation}
Fix an $\eps>0$ and a finite set $R\sub[0,1]$. We obtain that for every $r\in R$ there exists some $\alpha_r$ such that for every $\beta\geq\alpha_r$  we have $\pm (\1_A -x_\beta)(r)\leq y_\beta(r)\leq\eps$. This implies $x_\beta(r)\in[1-\eps,1+\eps]$.
Let $\tilde{\alpha}_R\in\mathbb{A}$ be an upper bound of the finite set $\left\{\alpha_r\mid|r\in R\right\}$. Then for every $r\in R$ and for every $\beta\geq \tilde{\alpha}_R$ we have $x_\beta(r)\in[1-\eps,1+\eps]$.
We conclude that for every finite set $R\sub[0,1]$ we have
\begin{equation}\label{ex1.eq5}
\exists\tilde{\alpha}_R\in\mathbb{A}\hs \forall \beta\geq \tilde{\alpha}_R\hs \forall r\in R\colon x_\beta(r)\geq 1-\eps.
\end{equation}
Fix an $\alpha\in\mathbb{A}$ and let $R:=N_\alpha$ as in \eqref{ex1.eq1}. Due to \eqref{ex1.eq5}
for every $\beta\geq \tilde{\alpha}_{N_\alpha}$ we have
\begin{equation*}
x_\beta(2) = \sum_{r\in N_\beta} x_\beta(r)\geq \t{card}(N_\beta)\cdot (1-\eps).
\end{equation*}
Next we show that there exists a subsequence $(x_n)_n$ of the net $(x_\alpha)_\alpha$ such that
\begin{equation}\label{ex1.eq6}
\forall n\in\NN\hs \exists\gamma\in\mathbb{A}\colon \t{card}(N_\gamma)\geq n.
\end{equation}
Then for every $n\in\NN$ there exists some $\tilde{\alpha}_{N_\gamma}\in\mathbb{A}$ such that for every $\beta\geq\tilde{\alpha}_{N_\gamma}$ we have $x_\beta(2)\geq n\cdot (1-\eps)$. This then implies that the net $(x_\alpha(2))_\alpha$ diverges, i.e.\ the ideal $I$ is o-closed.

To show that there exists a subsequence $(x_n)_n$ of $(x_\alpha)_\alpha$ with \eqref{ex1.eq6}, fix an $n\in\NN$ and let $N\sub[0,1]$ be such that $\t{card}(N)=n$. By \eqref{ex1.eq4} for every $\eps>0$ and every $r_k\in N$, where $k=1,\ldots,n$, there exists some $\alpha_k$ such that for every $\beta\geq\alpha_k$ we have $y_\beta(r_k)\leq\eps$. Fix an $\eps\in\hs]0,1[$ and let $\gamma$ be an upper bound of $\left\{\alpha_k\mid|k=1,\ldots,n\right\}$. Then for every $r_k\in N$ we have $y_\gamma(r_k)\leq\eps$.
As $x_\alpha\xrightarrow{o}\1_A$, it follows $\pm(\1_A-x_\gamma)(r_k)=\pm(1-x_\gamma(r_k))\leq y_\gamma(r_k)\leq\eps$. Due to $\eps\in\hs]0,1[$ this yields $x_\gamma(r_k)\neq 0$ for every $r_k\in N$. That is, for $x_\gamma$ we have $N\sub N_\gamma$. To sum up, there exists some $\gamma$ with $\t{card}(N_\gamma)\geq n$. Therefore the net $(x_\alpha(2))_\alpha$ of reals diverges and thus the ideal $I$ is o-closed.
\end{example}
\smallskip

From Corollary~\ref{1.4} we know that in pervasive Archimedean pre-Riesz spaces every directed band is s-closed. The next example demonstrates that this is not true in arbitrary pre-Riesz spaces.
\begin{example}\label{bands_not_s-closed}
\textit{In an Archimedean pre-Riesz space a directed band need not be s-closed.}

This example is similar to Example~\ref{ex1}.
Consider on the interval $A:=[-1,1]\cup\left\{2\right\}\sub\RR$ the vector lattice $B(A)$ of bounded functions on $A$. Let $X$ be the linear subspace of $B(A)$ which is generated as the linear hull of the characteristic functions $\1_{A}$ and $\1_{\left\{a,2\right\}}$ for every $a\in [0,1]$ and the function
\[h(t):= \begin{cases} 2t +2 \quad \t{ for } t\in[-1,-\frac{1}{2}] \\
					 -2t \quad\quad \t{ for } t\in[-\frac{1}{2},0] \\
					 0 \quad\quad\quad\t{ elsewhere.}\end{cases}\]
Endowed with pointwise order, $X$ is a directed and Archimedean ordered vector space and therefore a pre-Riesz space. Similarly to Example~\ref{ex1} one can see that $X$ is not a vector lattice.
We show that the following linear subspace of $B(A)$ is a vector lattice cover of $X$:
\[Y:=\left\{c\1_{A} + \lambda h + \sum_{k=1}^{n}\lambda_k\1_{\left\{a_k\right\}}\in B(A)\mid| \lambda,c,\lambda_k\in\RR,\hs a_k\in[0,1]\cup\left\{2\right\} \right\}.\]
First we show that $X$ is order dense in the ordered vector space $Y$. To that end, let $f\in Y$ with $f=c\1_{A} + \lambda h + \sum_{k=1}^{n}\lambda_k \1_{\left\{a_k\right\}}$, where $a_k \in [0,1]\cup\left\{2\right\}$ are pairwise different elements. We can assume $f\geq 0$ or else translate $f$ by adding the function $n\1_{A}\in X$ for a sufficiently large $n\in\NN$.
To write $f$ as an infimum of a subset of $X$, we will define a family of functions. Let $a, b \in[0,1]$ such that $a\neq b$ and let
\[f_a:=2||f||_\infty \1_A \hs +\hs \lambda h \hs +\hs (f(a)-2||f||_\infty)\1_{\left\{a,2\right\}} \hs +\hs (f(2)-f(a))\1_{\left\{b,2\right\}}.\]
Then $f_a\in X$ and we have $f_a(a) = f(a)$, $f_a(2) = f(2)$,
\begin{align*}
& f_a(b) = 2||f||_\infty+f(2)-f(a)\geq 2||f||_\infty-f(a)\geq ||f||_\infty\geq f(b) \hs\t{ and}\\
& f_a(t)=2||f||_\infty + \lambda h(t)\geq f(t) \hs\t{ for }\hs t\in A\ohne\left\{a,b,2\right\}.
\end{align*}
Moreover, the function
\[f_{-1} := c\1_A\hs +\hs \lambda h \hs +\hs \sum_{k=1}^{n}\lambda_k\1_{\left\{a_k,2\right\}} + \max\left\{f(2),2\sum_{i=k}^{n}|\lambda_k|\right\} \1_{\left\{0,2\right\}}\]
belongs to $X$ and we have $f_{-1}(t) = f(t)$ for $t\in A \ohne\left\{0,2\right\}$,
\begin{align*}
& f_{-1}(2)=c+ \sum_{i=k}^{n}\lambda_k +\max\left\{f(2), 2\sum_{i=k}^{n}|\lambda_k|\right\} \geq c+ \max\left\{f(2), \sum_{i=k}^{n}|\lambda_k|\right\} \geq f(2),\\
& f_{-1}(0) = c + \gamma + \max\left\{f(2),2\sum_{i=k}^{n}|\lambda_k|\right\} \geq c +\gamma = f(0),\\
&\qquad\qquad\t{ where }\hs \gamma = \lambda_k \hs\t{ in case }\hs a_k =0, \hs\t{ and otherwise } \gamma=0.
\end{align*} 

Thus we found for every $a\in [0,1]\cup\left\{-1\right\}$ a function $f_a\in X$ such that $f\leq f_a$. Moreover, for every $t\in[-1,1]$ we have $f_{-1}(t) = f(t)$  and for $a\in [0,1]$ we have $f_a(2)=f(2)$. It follows
$f=\inf\left\{f_a\in X\mid| a\in[0,1]\cup\left\{-1\right\}\right\}$ in $Y$,
i.e.\ $X$ is order dense\footnote{Note that though the order in $Y$ is pointwise, the supremum and infimum are \emph{not} pointwise, in general. The details can be seen below in the computation of the supremum of the two functions $g_1$ and $g_2$.} in $Y$.

We show that $Y$ is a vector lattice. To that end, we simplify the computations by considering the following subspace $Z$ of $B(A)$, consisting of restrictions of elements of $Y$ to $[-1,0[$, i.e.\
\[Z=\left\{f\in B(A) \mid| f=c\1_{[-1,0[} + \lambda h\right\}.\]
We show that $Z$ is a vector lattice. Let $g_1, g_2 \in Z$ be given by $g_1= c_1 \1_{[-1,0[} + \lambda_1h$ and $g_2 = c_2 \1_{[-1,0[} + \lambda_2h$ for some $c_1,c_2,\lambda_1,\lambda_2\in\RR$. We can assume $g_1$, $g_2$ to be positive.
We show that the function $s\in Z$, given by
\[s:=\max\left\{c_1,c_2\right\}\1_{[-1,0[} + \left(\max\left\{c_1+\lambda_1, c_2+\lambda_2\right\} - \max\left\{c_1,c_2\right\}\right)\cdot h,\]
is the supremum of $g_1$ and $g_2$.
Observe that $s(-1) = \max\left\{c_1,c_2\right\} \geq c_1 = g_1(-1)$. Further on,
$s(-\frac{1}{2})
 = \max\left\{c_1+\lambda_1, c_2+\lambda_2\right\} \geq c_1 +\lambda_1 = g_1(\textstyle{-\frac{1}{2}})$
and analogously $s(-1) \geq g_2(-1)$ and $s(-\frac{1}{2}) \geq g_1(-\frac{1}{2})$. Since $g_1,g_2$ and $s$ are affine linear between $t=-1$ and $t=-\frac{1}{2}$ and symmetric, it follows $g_1,g_2\leq s$.
Let $u\in Z$ be an upper bound of $g_1$ and $g_2$. We establish that $s\leq u$.
Due to $u\geq g_1, g_2$ we have $u(-1)\geq g_1(-1) = c_1$ and $u(-1)\geq g_2(-1) = c_2$, i.e.\
\begin{align*}
u(-1)\geq\max\left\{c_1,c_2\right\} =s(-1).
\end{align*}
Moreover, $u(-\frac{1}{2}) \geq g_1(-\frac{1}{2}) = c_1+\lambda_1 h(-\frac{1}{2}) = c_1+\lambda_1$ and $u(-\frac{1}{2}) \geq g_2(-\frac{1}{2}) =c_2+\lambda_2$,
i.e.\ $u(-\tfrac{1}{2}) \geq \max\left\{c_1+\lambda_1, c_2+\lambda_2\right\}= s(-\tfrac{1}{2})$.
The functions $s$ and $u$ are affine linear in $[-1,-\tfrac{1}{2}]$ and symmetric. Thus $u(-1)\geq s(-1)$ and $u(-\tfrac{1}{2})\geq s(-\tfrac{1}{2})$ imply $u\geq s$.
We obtain $s=g_1\Sup g_2$ in $Z$. That is, $Z$ is a vector lattice. Observe that, in general, $s$ is not a pointwise supremum of $g_1$ and $g_2$ in $Z$.

From this it immediately follows that $Y$ is a vector lattice. Indeed, consider the supremum of two elements in $Y$ as the supremum in $Z$, but additionally pointwise for $t\in [0,1] \cup \left\{2\right\}$. Notice that $X$ is not pervasive in $Y$, since there does not exist an $x\in X$ with $0\leq x\leq \1_{\left\{2\right\}} \in Y$.

Now let $B\sub X$ be the ideal generated by all functions of the form $\1_{\left\{a,2\right\}}$, where $a\in [0,1]$. Notice that $B$ is directed. We show that $B$ is a band. To that end, we first compute $B^{\t{d}}$. Recall that for $x_1,x_2\in X$ we have $x_1 \perp x_2$ in $X$ if and only if $x_1\perp x_2$ in $Y$. Let $x\in X$ and $\1_{\left\{a,2\right\}}\in B$. Let $r \in Y$ be the infimum of $|x|$ and $\1_{\left\{a,2\right\}}$. It follows for every $t\in[-1,1]$ that $r(t) \leq |x|(t)$ and $r(t)\leq\1_{\left\{a,2\right\}}(t)$. Therefore for every $t\in[0,1]\ohne\left\{a\right\}$ we have $r(t) \leq 0$. Since $0$ is a lower bound of $|x|$ and $\1_{\left\{a,2\right\}}$, it is $0 \leq r(t)$, i.e.\ $r(t) = 0$ for $t\in[0,1]\ohne\left\{a\right\}$. The only elements in $X$, which are zero on $[0,1]\ohne\left\{a\right\}$ are linear combinations of $h$ and $\1_{\left\{a,2\right\}}$. Since $a$ is arbitrary in $[0,1]$, it follows that multiples of $h$ are the only elements in $X$ which are disjoint to all elements of $B$, i.e.\ in $X$ we have $B^{\t{d}} = \left\{\lambda h \mid| \lambda\in\RR \right\}$. With a similar argumentation we obtain $B\dd= B$, therefore $B$ is a band.

We show that $B$ is not s-closed. Let namely $s:=\1_A -h \in X_+\ohne B$. Consider the set $M:=B\cap[0,s]$. We show that $\sup M$ exists in $X$ and $\sup M=s$. For every $a\in [0,1]$ it is $\1_{\left\{a,2\right\}}\in M$, therefore $M\neq\varnothing$. We establish that $s$ is the least upper bound of $M$. Let $r\in X$ be an upper bound of $M$, such that
\begin{align}\label{ex_band_not_sclosed_3}
r = c\1_A +\lambda h +\sum_{k=1}^n \lambda_k \1_{\left\{a_k,2\right\}}
\end{align}
with pairwise different elements $a_k\in [0,1]$.
For every $t\in [0,1]$ we have $\1_{\left\{t,2\right\}}\in M$. It follows
$r(t)\geq \1_{\left\{t,2\right\}}(t)=1$ for every $t\in [0,1]$ and $r(2)\geq 1$. Since the sum in \eqref{ex_band_not_sclosed_3} contains only finitely many summands, this leads to $c\geq 1$. For $t\in[-1,0[$ we have $r(t) \geq 0$ and therefore $\lambda \geq -1$. To sum up, we have
\begin{align*}
r(t)\geq 1 \t{ for } t\in[0,1]\cup\left\{2\right\} \quad\t{ and }\quad c\geq 1,\hs \lambda \geq -1.
\end{align*}
For $t\in[0,1]\cup\left\{2\right\}$ it follows $s(t) = 1 \leq r(t)$. Moreover, $s(-1)=1 \leq c = r(-1)$ and $s(-\frac{1}{2}) = 0 \leq 1 + \lambda \leq c+\lambda = r(-\frac{1}{2})$. Since $s$ and $r$ are affine linear on $[-1,-\frac{1}{2}]$ and since $h$ is symmetric, we obtain 
$s(t) \leq r(t)$ for $t\in[-1,0[$. This establishes $s\leq r$. Consequently, $s=\sup M=\sup(B\cap[0,s])$ in $X$. However, $s\notin B$, i.e.\ $B$ is not s-closed.
\end{example}

\section{Additional Examples}
We provide two further examples concerning bands in pre-Riesz spaces.
Since for the definition of an s-closed ideal only the positive part of the ideal is important, it is a natural question to ask, whether every s-closed ideal is automatically directed.
Example~\ref{Namioka3} shows that this is not the case, even under strong additional conditions on the space.
Example~\ref{3} provides an instance of a pervasive pre-Riesz space without RDP in which every o-closed directed ideal is a band. This suggests the conjecture that in Theorem~\ref{2.1} the condition of RDP might be dropped.
\begin{example}\label{Namioka3}
\textit{In an Archimedean pervasive pre-Riesz space with RDP and an order unit an s-closed band need not be directed.}

We return to Examples \ref{Namioka1} and \ref{Namioka2}. There we considered the pervasive pre-Riesz space $X=\left\{f\in C([-1,1]) \mid| f(0) = \frac{f(-1)+f(1)}{2}\right\}$ which has the RDP and an order unit. In Example~\ref{Namioka1} we introduced the ideal $I:=\left\{ f\in X \mid| f([-\textstyle{\frac{1}{2}},0]\cup\left\{-1,1\right\}) = 0 \right\}$ and established in Example~\ref{Namioka2} that the band
$B:=I\dd=\left\{f\in X \mid| f([-\textstyle{\frac{1}{2}},0]) = 0\right\}$ is not directed.
We show that $B$ is s-closed. To that end, let $z\in X_+$ be such that $z=\sup\left(B\cap[0,z]\right)$. We establish $z\in B$.
First, observe that for $f\in B_+$ we have $f(-1), f(1)\geq 0$ and from $0=f(0)=\frac{f(-1)+f(1)}{2}$ it follows $f(-1)=f(1)=0$, i.e.\ $f\in I_+$. That is, $I_+ = B_+$. It follows
\[z=\sup\left(B\cap[0,z]\right) = \sup\left(I\cap[0,z]\right).\]
Since by Example~\ref{Namioka1} the ideal $I$ is s-closed, we obtain $z\in I\sub B$. Thus the band $B$ is s-closed.
\end{example}

In Theorem~\ref{2.1} we established that every o-closed directed ideal $I$, with $I\dd$ directed, is a band if the pre-Riesz space is pervasive and has the RDP.
The following example suggests that the condition of RDP in Theorem~\ref{2.1} might be dispensable.
\begin{example}\label{3}
\textit{A pervasive pre-Riesz space without RDP in which every o-closed directed ideal is a band.}

Consider the vector lattice of piecewise affine continuous functions (finitely many pieces) on the interval $[-1,1]$, i.e.\ 
\[\t{PA}[-1,1] := \left\{ f\in C[-1,1] \mid| f \t{ piecewise affine, continuous} \right\}.\]
Let $q\in C[-1,1]$, $q(t) = t^2$. The ordered vector space
\[X:=\t{PA}[-1,1]\oplus\left\{\lambda q \mid| \lambda \in\RR\right\}=\left\{ f+\lambda q \mid|f \in \t{PA}[-1,1], \lambda\in\RR  \right\}\]
is a subspace of $C[-1,1]$ and has an order unit, e.g.\ $\1_{[-1,1]}$. Thus $X$ is Archimedean and directed and therefore pre-Riesz.
Observe that $\t{PA}[-1,1]$ is a vector sublattice of $C[-1,1]$ and by \c[VIII.4.7]{Werner} dense in $C[-1,1]$ with respect to the supremum norm.
Due to $\t{PA}[-1,1]\sub X\sub C[-1,1]$ every element of $C[-1,1]$ can be approximated pointwise from above by elements of $X$. Therefore $X$ is order dense in $C[-1,1]$. Moreover, $X$ is pervasive, since the subspace $\t{PA}[-1,1]$ of $X$ is pervasive. 

We show that $X$ does not have the RDP. Suppose that $X$ has the RDP and let
\[x_1(t)= \begin{cases}0 &\quad t\in[-1,0]\\ t &\quad t\in\hs ]0,1] \end{cases}
\quad \t{ and } \quad
x_2(t)= \begin{cases}-t &\quad t\in[-1,0]\\ 0 &\quad t\in\hs ]0,1]. \end{cases}\]
Then $x_1, x_2 \in X_+$ and $q \leq x_1 + x_2$.
We show that there are no elements $a_1, a_2 \in X$ with $0\leq a_1 \leq x_1$, $0\leq a_2 \leq x_2$ and $q = a_1 + a_2$. Indeed, let $a_1 \in X$ be such that $0\leq a_1 \leq x_1$ and $a_1 = f + \lambda q$ with $f\in \t{PA}[-1,1]$, $\lambda \in \RR$ and $q= a_1+a_2$. Then $a_1(t) = 0$ for every $t\in [-1,0]$. This implies $\lambda = 0$ and so, $a_1 \in \t{PA}[-1,1]$. Similarly $a_2 \in \t{PA}[-1,1]$. Altogether, we have $a_1 + a_2 = q \in \t{PA}[-1,1]$, which is a contradiction. Therefore $X$ does not have the RDP.

Next we show that every directed o-closed ideal in $X$ is a band.
For an ideal $A$ in $X$ let $Z(A)=\left\{t\in[-1,1]\mid| x(t)=0 \t{ for every } x\in A\right\}$ be the set of zeroes of $A$. Notice that every element in $X$ is a linear combination $\lambda q+p$ of the quadratic function $q$, where $\lambda\in\RR$, with an element $p\in\t{PA}[-1,1]$. First we observe the following facts.

\textbf{Claim 1:} \textit{Let $p\in\t{PA}[-1,1]$ and $\lambda\neq 0$. Then $\lambda q +p$ has finitely many zeroes.} This follows immediately from the fact that $q$ is quadratic and $p$ is piecewise affine.
\textbf{Claim 2:} \textit{Let $A\sub X$, $A\neq\left\{0\right\}$, be an ideal with $Z(A)$ infinite. Then for every $p\in\t{PA}[-1,1]$ and $\lambda\neq 0$ we have $\lambda q +p\notin A^{\t{d}}$.}
Indeed, since $A\neq\left\{0\right\}$, there exists some $a\in A$, $a\neq 0$. We can assume that there is an $r\in [-1,1]$ with $a(r)> 0$. Due to the continuity of $a$ there is an open interval $U_r$ around $r$ such that $a(t)>0$ for every $t\in U_r$. From $a\perp A^{\t{d}}$ it follows that every $x\in A^{\t{d}}$ is zero on $U_r$. That is, $U_r\sub Z(A^{\t{d}})$. As for $\lambda\neq 0$ the function $\lambda q +p$ has finitely many zeroes, it follows $\lambda q +p\notin A^{\t{d}}$.

Next, to establish that every o-closed directed ideal in $X$ is a band, we consider four cases. In some cases we establish instead that if an ideal is a band, then it is not o-closed. Let $I$ be an ideal in $X$.

\textit{Case 1}: Let first $Z(I)$ be a singleton, i.e.\ $Z(I)=\left\{r\right\}$ for an $r\in[-1,1]$. Then due to $I\dd=X$ we have that $I$ is not a band. We show that $I$ is not order closed. To that end we find a sequence $(f_n)_{n\in\NN}$ in $I$ which o-converges to $\1_{[-1,1]}\in X\ohne I$. Define for every $n\in\NN$ a neighbourhood $U_n:=[r-\tfrac{1}{n},r+\tfrac{1}{n}]\cap[-1,1]$ of $r$. Consider for every $n\in\NN$ the mapping $f_n$ which is defined by $f_n(t)=0$ for $t\in U_{n+1}$ and by $f_n(t)=1$ for $t\in [-1,1]\ohne U_n$. The mapping $f_n$ can be extended continuously by an affine piece on $\left]r-\tfrac{1}{n},r+\tfrac{1}{n+1}\right[\cap[-1,1]$ and  on $\left]r-\tfrac{1}{n+1},r+\tfrac{1}{n}\right[\cap[-1,1]$, respectively (provided these sets are not empty). Let $n\in\NN$ be fixed. Clearly, $f_n\in \t{PA}[-1,1]$. We show $f_n\in I$. Then for every $n\in\NN$ we can set $y_n:=\1_{[-1,1]} - f_n\in X$ and due to $f_n\uparrow$ obtain $\pm(\1_{[-1,1]} - f_n)\leq y_n\downarrow 0$, which yields $f_n\xrightarrow{o} \1_{[-1,1]}\notin I$. This then implies that $I$ is not order closed.
To establish that for a fixed $n\in \NN$ we have $f_n\in I$, let $s\in [-1,1]\ohne\left\{r\right\}$. Then there exists a function $x_s\in I$ with $x_s(s)\neq 0$. We can assume that $x_s(s)>0$ and, since $I$ is a linear subspace of $X$, we can choose $x_s$ to satisfy $x_s(s)=2$. Due to $x_s$ being continuous, it follows that there is an open interval $U_s$ around $s$ such that $x_s(t)\geq 1$ for every $t\in U_s$.
Thus there exists a positive function $g_s\in \t{PA}[-1,1]$ such that $g_s$ has its support in $U_s$ and there is an open interval $V_s$ around $s$ such that $g_s(t)=1$ for $t\in V_s$ and $0\leq g_s(t) \leq 1$ for $t\in U_s$. Since for every $t\in[-1,1]$ we have $0\leq g_s(t)\leq |x_s(t)|$, it follows $g_s\in I$.
For an $\eps\in\left]0,1\right[$ let $V_0:=\left]-\eps,\eps\right[$. Then the set $\left\{V_s\mid|s\in[-1,1]\right\}$ is an open cover of the compact set $[-1,1]$ in the topological space $[-1,1]$ and thus has a finite subcover.
Let $S\sub[-1,1]$ be a finite set such that $\left\{V_s\mid| s\in S\right\}$ is a finite subcover of $[-1,1]$. Then the piecewise affine function $g_{\eps}:=\sum_{s\in S\ohne\left\{0\right\}} g_s$ belongs to $I_+$ and we have $g_\eps(t)\geq 1$ for every $t\in[-1,1]\ohne\left]-\eps,\eps\right[$. Choose $\eps>0$ such that $\left[r-\tfrac{1}{n+1},r+\tfrac{1}{n+1}\right]\sub \left]-\eps,\eps\right[$. Then it follows $0\leq f_n\leq g_\eps$, i.e.\ $f_n\in I$.

\textit{Case 2}: Let $Z(I)$ be finite. Then $I$ is not a band, as $I\dd=X$. Since every point in $Z(I)$ is an isolated point in $[-1,1]$, by a similar argumentation as in the previous case we find a net in $I$ which order converges to $\1_{[-1,1]}$. That is, $I$ is not o-closed. 

\textit{Case 3}: Let $Z(I)$ be infinite and $I\dd\neq X$. We show that if $I$ is o-closed, then $I$ is a band. To that end, we first establish that $I$ is a band in $\t{PA}[-1,1]$ if and only if $I$ is a band in $X$.
Let $\lambda q+ p\in X$ with $p\in\t{PA}[-1,1]$ and $\lambda\neq 0$.
Since $Z(I)$ is infinite, by Claim 1 it follows $\lambda q +p\notin I$.
The case $I=\left\{0\right\}$ is clear, therefore we can assume $I\neq\left\{0\right\}$.
By Claim 2 we obtain $\lambda q +p\notin I^{\t{d}}$.
Due to $I\dd\neq X$ we have $I^{\t{d}}\neq\left\{0\right\}$ and $Z(I^{\t{d}})$ is infinite.
Then Claim 2 yields $\lambda q +p\notin I\dd$.
To sum up, we established that $I\cap \t{PA}[-1,1]=I$, $I^{\t{d}}\cap \t{PA}[-1,1]=I^{\t{d}}$ and $I\dd\cap \t{PA}[-1,1]= I\dd$. That is, the quadratic part of $X$ is not involved into the procedure of taking the disjoint complement. It follows that $I$ is a band in $\t{PA}[-1,1]$ if and only if $I$ is a band in $X$. 
Next we show that if $I$ is o-closed in $X$, then $I$ is o-closed in $\t{PA}[-1,1]$. Since in the vector lattice $\t{PA}[-1,1]$ bands are precisely the o-closed ideals, this then establishes that if $I$ is o-closed in $X$, then $I$ is a band in $X$.
Let $I$ be not o-closed in $\t{PA}[-1,1]$. We need to show that $I$ is not o-closed in $X$. As $I$ is not o-closed in $\t{PA}[-1,1]$, there exists some $x\in\t{PA}[-1,1]\ohne I$ and a net $(x_\alpha)_\alpha$ in $I$ such that $x_\alpha\xrightarrow{o} x$ in $\t{PA}[-1,1]$. Due to $\t{PA}[-1,1]$ being order dense in $X$, by \c[Proposition~5.1~(ii)]{2} it follows $x_\alpha\xrightarrow{o} x$ in $X$. That is, $I$ is not o-closed in $X$.
We conclude that if $I$ is an o-closed ideal in $X$, then $I$ is a band.

\textit{Case 4}: Let $Z(I)$ be infinite and $I\dd= X$. Suppose first $I^{\t{d}}\neq\left\{0\right\}$. 
Then for $\lambda\neq 0$ and $p\in\t{PA}[-1,1]$ Claim 2 yields $\lambda q +p\notin I\dd =X$, a contradiction. Therefore we have $I^{\t{d}}=\left\{0\right\}$. 
We can assume that $Z(I)$ does not contain an isolated point. Indeed, if $Z(I)$ has an isolated point, then $I$ is not a band, and by a similar argumentation to the first case not o-closed. Moreover, $Z(I)$ has no inner points. If, on the contrary, $z$ is an inner point of $Z(I)$, then there is an open neighbourhood $U_z$ of $z$ with $U_z\sub Z(I)$. We can find a continuous piecewise affine function $a\in X$, $a\neq 0$, with support in $U_z$. Then $a\perp I$, a contradiction to $I^{\t{d}}=\left\{0\right\}$.

To sum up, we have the following conditions: $Z(I)$ is infinite, $I\dd= X$, $I^{\t{d}}=\left\{0\right\}$, $Z(I)$ consists of accumulation points and has no inner points. The ideal $I$ is not a band. Indeed, suppose $I=I\dd=X$, then $Z(I)=Z(X)=\varnothing$, a contradiction. We establish that $I$ is not o-closed. 
To that end, we first show that $Z(I)$ is contained in a Cantor set $C$. For the real numbers $a_{(0)}:=\inf Z(I)$ and $a_{(1)}=\sup Z(I)$ we have $a_{(0)}<a_{(1)}$. Since $Z(I)$ has no inner points, there exists some $m_{(0,1)} \in\left]a_{(0)},a_{(1)}\right[$ with $m_{(0,1)}\notin Z(I)$. Thus there is a function $x^{(1)}\in I$ such that $x^{(1)}(m_{(0,1)})\neq 0$. It follows that there is an open interval $U_{1}$ around $m_{(0,1)}$ such that $x^{(1)}(t)\notin Z(I)$ for every $t\in U_{1}$. We can choose the interval $U_{1}$ to be maximal in respect to the inclusion.
That is, since $Z(I)$ is infinite and does not contain isolated points, there exist $a_{(0,1)}, a_{(1,0)} \in Z(I)$ with $a_{(0)}\neq a_{(0,1)}$ and $a_{(1)}\neq a_{(1,0)}$ such that $x^{(1)}(t)\notin Z(I)$ for every $t\in U_{1}:=\left]a_{(0,1)}, a_{(1,0)}\right[$. 
We consider the next two intervals $[a_{(0)},a_{(0,1)}]$ and $[a_{(1,0)}, a_{(1)}]$. Since $Z(I)$ has no inner points, there exist $m_{(0,0,1)} \in\left]a_{(0)},a_{(0,1)}\right[$ and $m_{(1,0,1)} \in\left]a_{(1,0)}, a_{(1)}\right[$ with $m_{(0,0,1)}, m_{(1,0,1)}\notin Z(I)$.
We can proceed as above and remove from the two intervals $[a_{(0)},a_{(0,1)}]$ and $[a_{(1,0)}, a_{(1)}]$ a maximal open subinterval of $[-1,1]\ohne Z(I)$ around $m_{(0,0,1)}$ and $m_{(1,0,1)}$. We obtain four intervals and can proceed by induction to obtain a Cantor set $C$. As $Z(I)$ does not belong to the removed open intervals, we obtain that $Z(I)\sub C$. We can number the removed intervals (starting with $U_1$) and the appropriate functions in $I$ which do not vanish on these respective intervals (starting with $x^{(1)}$), e.g.\ in the order in which the intervals are removed. Denote the intervals by $U_n$ and the associated functions by $x^{(n)}$, where $n\in\NN$.

We proceed to establish that $I$ is not o-closed. 
To that end, we consider the following function $f\notin I$ and show that there is a sequence in $I$ which o-converges to $f$:
\[f(t)= \begin{cases}
	\tfrac{1}{a_{(0,1)}-a_{(0)}}t-\tfrac{1}{a_{(0,1)}-a_{(0)}}a_{(0)} &\quad t\in[a_{(0)},a_{(0,1)}],\\
	1 &\quad t\in U_1=\left]a_{(0,1)},a_{(1,0)}\right[,\\
	\tfrac{1}{a_{(1,0)}-a_{(1)}}t-\tfrac{1}{a_{(1,0)}-a_{(1)}}a_{(1)} &\quad t\in[a_{(1,0)},a_{(1)}],\\
	0 &\quad t\in[-1,1]\ohne[a_{(0)},a_{(1)}].
\end{cases}\]
Clearly, $f\in X$. Considering only the interval $U_{1}\sub[-1,1]$, we can approximate the constant one function $\1_{U_1}$ (which does not belong to $X$) from below by a sequence $(x_m^{(1)})_{m\in\NN}$ of elements in $I$ such that $x_m^{(1)}\uparrow$. Moreover, we can approximate not only constant functions, but similarly functions $f_{\upharpoonright U_n}$ which are restricted to the open intervals $U_n$, for every $n\in\NN$, respectively. Let for every $n\in\NN$ the sequence $(x_m^{(n)})_{m\in\NN}$ approximate $f_{\upharpoonright U_n}$ such that $x_m^{(n)}\uparrow_m $.
Then for the piecewise affine function $s_n:=\sum_{k=1}^n x_n^{(k)}$ we have $s_n\in I$ and $s_n\uparrow$. Clearly, $f$ is an upper bound of $s_n$. Moreover, we have
\begin{equation}\label{3.eq1}
\forall t\in [-1,1]\ohne Z(I)\colon \lim_{n\rightarrow\infty}s_n(t) = f(t).
\end{equation} 
Fix an element $z\in Z(I)$. Recall that $Z(I)$ does not contain intervals. Thus for every neighbourhood $U(z)$ of $z$ there exists some $t\in U(z)\ohne Z(I)$. Therefore there is a sequence $(t_j)_{j\in\NN}$ in $[-1,1]\ohne Z(I)$ such that $t_j\rightarrow z$. Since $f$ is continuous, it follows $|f(t_j)-f(z)|\rightarrow 0$. 
Due to \eqref{3.eq1} for every $j\in\NN$ we have $|s_n(t_j)-f(t_j)|\xrightarrow{n\rightarrow\infty} 0$. This yields
\begin{align*}
|s_n(t_j)-f(z)| &= |s_n(t_j)-f(t_j)+f(t_j)-f(z)|\leq\\
&\leq |f(t_j)-f(z)|+ |s_n(t_j)-f(t_j)|\xrightarrow{j,n\rightarrow\infty} 0.
\end{align*}
For every upper bound $u\in\left\{s_n\mid|n\in\NN\right\}^u$ we have $s_n(t_j)\leq u(t_j)$ for every $j\in\NN$. As $s_n$ and $u$ are continuous, we obtain $f(z)\leq u(z)$. Clearly, for $t\in [-1,1]\ohne Z(I)$ due to $s_n\uparrow$ and \eqref{3.eq1} we have $f(t)\leq u(t)$. 
It follows that $f$ is the least upper bound of $\left\{s_n\mid|n\in\NN\right\}^u$, i.e.\ $s_n\uparrow f\notin I$. This establishes that $I$ is not o-closed. 

To sum up, in $X$ we have either bands which are o-closed, or ideals which are not o-closed.
\end{example}

\begin{numremark}
\begin{enumerate}
\item Example~\ref{3} provides a simple instance of a pre-Riesz space which is pervasive, but does not have the RDP. The idea of this example can be used to construct other spaces of this kind. Until recently the only known example for a pervasive pre-Riesz space without RDP was the space $L^r(l_0^\infty)$ of regular operators on the space $l_0^\infty$ of eventually constant real sequences, see \c{AbrWick1991} and \c{3}.
\item Let $X$ and $X_0$ be pre-Riesz spaces such that $X_0\sub X$ and let $Y$ be a vector lattice cover of $X$ and $X_0$. Clearly, if $X_0$ is pervasive, then $X$ is likewise pervasive. Example~\ref{3} shows that a similar statement for RDP is not true: The space $X_0$ has the RDP, but $X$ does not. Here, $X_0=\t{PA}[-1,1]$ is even a vector lattice.
\end{enumerate}
\end{numremark}

\textbf{Acknowledgments:} The author thanks Anke Kalauch for her valuable comments and Onno van Gaans for his idea to investigate s-closed ideals and his contribution to Example~\ref{3}.


\bibliographystyle{plain}

\end{document}